\documentclass[12pt]{amsart}
\usepackage{amsmath,amsthm,amsfonts,amssymb, mathrsfs,color}

\usepackage{graphicx}
\usepackage{url}
\usepackage{color}
\definecolor{maroon}{rgb}{.69,.188,.376}
\definecolor{darkgreen}{rgb}{0,.5,0}
\definecolor{darkblue}{rgb}{0,0,.5}
\definecolor{magenta}{rgb}{1,0,1}

\usepackage[mathscr]{euscript}		

\usepackage{dsfont}

\usepackage{psfrag}			
\usepackage[colorlinks=true]{hyperref}
\hypersetup{pdftex, colorlinks=true, linkcolor=maroon, citecolor=maroon,
  filecolor=blue,urlcolor=blue}

\calclayout

\numberwithin{equation}{section}

\usepackage[mathscr]{euscript}		

\newtheorem{thm}{Theorem}[section]
\newtheorem{lem}{Lemma}[section]
\newtheorem{prop}{Proposition}[section]

\newtheorem{defn}{Definition}[section]
\newtheorem{rem}{Remark}[section]
\theoremstyle{definition}

\newcommand{\Cov}{\text{Cov}}
\newcommand{\Var}{\text{Var}}

\numberwithin{equation}{section}

\newcommand{\be}{\begin{equation}}
\newcommand{\ee}{\end{equation}}

\newcommand{\bes}{\begin{equation*}}
\newcommand{\ees}{\end{equation*}}

\newcommand{\mP}{\mathbb{P}}

\newcommand{\R}{\mathbb{R}}
\newcommand{\Z}{\mathbb{Z}}

\newcommand{\N}{\mathbf{N}}

\newcommand{\g}{\textbf{g}}

\newcommand{\bean}{\begin{eqnarray*}}
\newcommand{\eean}{\end{eqnarray*}}

\newcommand{\mf}{\mathbf}
\newcommand{\msr}{\mathscr} 

\newcommand{\uv}{\mathbf u}
\newcommand{\W}{\mathbf W}
\newcommand{\D}{\mathbf D}
\newcommand{\vv}{\mathbf v}

\newcommand{\cl}{\mathscr C_1}
\newcommand{\cu}{\mathscr C_2}
\newcommand{\dlip}{\mathscr D}

\begin{document}

\title[Small ball and support theorems for SPDE]{Small ball probabilities and a support theorem for the  stochastic heat equation}
\author{Siva Athreya \and Mathew Joseph \and Carl Mueller}
\address{Siva Athreya\\ Statmath Unit\\ Indian Statistical Institute\\ 8th Mile Mysore Road\\ Bangalore 560059
} \email{athreya@isibang.ac.in}

\address{Mathew Joseph\\ Statmath Unit\\ Indian Statistical Institute\\ 8th Mile Mysore Road\\ Bangalore 560059
} \email{m.joseph@isibang.ac.in}

\address{Carl Mueller \\Department of Mathematics, University of Rochester Rochester, NY  14627
}
\email{carl.e.mueller@rochester.edu}


\begin{abstract}
We consider the following stochastic partial differential equation on
$t \geq 0, x\in[0,J], J \geq 1$ where we consider $[0,J]$ to be the
circle with end points identified:
\begin{equation*}
\partial_t{\mf u}(t,x) =\frac{1}{2}\,\partial_x^2 {\mf u}(t,x)
+ \mf g(t,x,\mf u) + {\mf \sigma}(t,x, {\mf u})\,\dot {\mf W}(t,x) ,
\end{equation*}
and $\dot {\mf W }(t,x)$ is 2-parameter $d$-dimensional vector valued
white noise and ${\mf \sigma}$ is function from $\R_+\times \R \times
\R^d$ to space of symmetric $d\times d$ matrices
which is Lipschitz in $\mf u$. We assume that $\sigma$ is uniformly
elliptic and that $\mf g$ is uniformly bounded.  Assuming that ${\mf
  u}(0,x) \equiv \mf 0$, we prove small-ball probabilities for the
solution $u$. We also prove a support theorem for solutions, when
${\mf u}(0,x)$ is not necessarily zero.
\end{abstract} 
\maketitle

\section{Introduction}
In this article we study small-ball probabilities and support 
theorems  for solutions to the stochastic heat equation given by
\begin{equation} \label{mainspde}
\partial_t{\mf u} (t,x) =\frac{1}{2}\partial_x^2{\mf u}(t,x) 
+ \mf g(t,x,\mf u) + {\mf \sigma}(t,x, {\mf u})\, \dot{\mf W}(t,x).
\end{equation}
where, $t \in \R_+, x \in \R$,  
$\dot {\mf W }(t,x)=(\dot {\mf W }_1(t,x),\ldots,\dot {\mf W }_d(t,x))$
 is $d$-dimensional space-time white noise, with $d \geq 1$ and  
${\mf \sigma}$ is a function from  $\R_+\times \R \times \R^d$ to space of symmetric $d\times d$ matrices. Assuming 
that $\sigma$ is Lipschitz in $\mf u$ and uniformly elliptic, $\mf g$ is 
uniformly bounded, and ${\mf u}(0,x)\equiv\mf 0$, 
our main result Theorem \ref{th:small-ball} provides upper and lower bounds for the small 
ball probabilities of the solution to (\ref{mainspde}). This result gives bounds on the probability that the profile 
${\mf u}(t,\cdot)$ {\it stays close} to the zero profile up to time $T$, see Theorem \ref{th:small-ball} for the precise statement. As a consequence of 
the above result we are able to prove a {\it support theorem}, which
provides similar bounds on the probability that the profile ${\mf u}(t,\cdot)$ stays close to a twice differentiable function up to time $T$,  see Theorem
\ref{th:support} for the precise statement.

Small ball problems have been well studied and have a long history in
probability theory, see \cite{LS01} for a survey.  More precisely, for
a stochastic process $X_t$ starting at 0, we are interested in the
probability that $X_t$ stays near its starting point for a long time,
that is,
$$P(\sup_{0\leq t\leq T}|X_t|<\varepsilon)$$ where $\varepsilon>0$ is
small. When $X_t$ is Brownian motion, small ball estimates follow from
the reflection principle or from the study of eigenvalues, among other
techniques.  Donsker and Varadhan \cite{DV75} obtained small ball
estimates for a wide class of Markov processes as a result of their
theory of large deviations of local time.  For other processes, the
complexity of the small ball estimates are well-known. Moreover, in
general small ball probabilities are usually harder to estimate than
the the large deviation probability that $X_t$ achieves unusually
large values, that is $$P(\sup_{0\leq t\leq T}|X_t|>\lambda)$$ for
large values of $\lambda$.

For some class of Gaussian processes small ball probabilities can be
determined.  Often these results are given in terms of metric entropy
estimates which are hard to explicitly compute \cite{KL93}.  One
exceptional case is the Brownian sheet, for which fairly sharp small
ball estimates are explicitly known, see Bass \cite{B88} and
Talagrand \cite{T94}. In \cite{spde-utah}, Xiao provides small ball
estimates for Gaussian processes that satisfy a certain condition 
which is related to the Gaussian concept of local nondeterminism. 
In \cite{LS01}, an overview of known results on Gaussian processes and
references on other processes such as fractional Brownian motion are
given.

There has not been much exploration of small ball probabilities in the
context of stochastic PDEs. Lototsky \cite{L17} has studied small ball
problems for a linear SPDE with additive white noise, where the
solution is a Gaussian process.  Martin \cite{Mar04} followed the
approach of Talagrand \cite{T94} to study the following stochastic
wave equation.
\begin{equation}
\label{eq:martin}
\partial_t^2u =\partial_x^2u+f(u)+g(t,x)\dot{W}(t,x)
\end{equation}
where $\dot{W}(t,x)$ is two-parameter white noise and $g(t,x)$ is a 
deterministic function.  Without $f(u)$, the solution $u$ would be a 
Gaussian process of the type studied by Bass \cite{B88} and Talagrand 
\cite{T94}. Although \eqref{mainspde} and \eqref{eq:martin} have similar
multiplicative noise terms, in our case the noise coefficient ${\mf
  \sigma}(t,x, {\mf u})$ depends on the solution, and this dependence
takes us away from Gaussian processes setting.

Small ball probabilities have many applications, for e.g. they play
key role in studying the small scale behavior of Gaussian processes,
such as the Hausdorff dimension of the range (see \cite{B88},
\cite{T95}). Another key application is the support theorem. For e.g.,
in the case of Brownian motion an application of Girsanov Theorem yields the
classical support theorem (See \cite[Proposition 6.5 and Theorem
  6.6 on pages 59-60]{B95}). Once we have obtained small probability
estimates for solution \eqref{mainspde} then we use the Girsanov
Theorem for SPDE to obtain a support theorem for solution.

We are now ready to state our main results.

\subsection{Main Results}
\label{sec:model}
For any ${\mf u} \in \R^d$ we shall  denote $\mid {\mf u} \mid$ to be 
the standard Euclidean norm on $\R^d$ and 
$\langle {\mf u}, {\mf v} \rangle$ denote the inner product between 
${\mf u}, {\mf v} \in \R^d$.  ${\mathcal M}_d(\R)$ will denote the 
space of symmetric $d\times d$ matrices over real numbers. Let 
$(\Omega, {\mathcal F}, {\mathcal F}_t, \mP)$ be a filtered 
Probability space on which $\dot{\mf W}=\dot{\mf W}(t,x)$ is a $d$-dimensional 
random vector whose components are i.i.d. two-parameter white noises 
adapted to ${\mathcal F}_t$.  

We consider vector-valued solutions $\mf u(t,x) \in \mathbb{R}^d$, to
the following stochastic heat equation (SHE) 
\begin{equation}
\label{eq:she}
\begin{split}
\partial_t {\mf u} (t,x) &=\frac{1}{2}\,\partial_x^2{\mf u}(t,x) 
+ \mf g(t,x,\mf u(t,x)) + {\mf \sigma}(t,x, {\mf u}(t,x))\,\dot {\mf  W}
(t,x).  \\
\mf u(0,x)&=\mf u_0(x)\equiv{\mf 0},  
\end{split}
\end{equation}
on the circle with $x\in [0,J]$ and endpoints identified, and for the 
function 
$\sigma : \R_+ \times \R\times \R^d \rightarrow {\mathcal M}_d (\R)$. 
We assume that $\mf g: \R_+ \times \R\times \R^d \rightarrow \R^d$ is 
uniformly bounded, $\sigma(t,x,\mf u)$ is Lipschitz continuous in the 
third variable, that is there is a constant $\mathscr D>0$ such that 
for all $t\ge 0, \, x\in [0,J],\,\mf u,\mf v \in \R^d$,
\begin{equation}\label{eq:sigma:lip:2}
\vert \sigma(t,x, \mf u) - \sigma(t,x,\mf v) \vert \le \mathscr D \vert \mf u - \mf v\vert. 
\end{equation}
We will further assume that the functions $\sigma$ is uniformly 
elliptic, that is there are constants
${\mathscr C}_1, {\mathscr C}_2, >0$ such that for all 
$t\ge0, \, x\in [0,J],\, \mf u \in \R^d,\, \mf y \in \R^d$ with
$|\mf y|=1$,
\begin{equation}
\label{eq:sigma:ue}
{\mathscr C}_1 \le \langle\mf y, \,\sigma(t,x, \mf u) \mf y\rangle \le  {\mathscr C}_2. 
\end{equation}
 The above implies that the matrix valued function $\sigma$ is positive definite everywhere (in particular invertible) and that all the eigenvalues of $\sigma$ are uniformly bounded above and away from $0$. 
 
As is usual in stochastic differential equations, \eqref{eq:she} is not 
well-posed as written, because the solution $\mf u$ is not 
differentiable and $\dot{\mf W}$ only exists as a generalized 
function.  We take \eqref{eq:she} to be shorthand for the mild form:
\begin{equation}
\label{eq:weak-form}
\begin{split}
\mf u(t,x)&=\int_{0}^{J}G(t,x-y)\mf u_0(y)
+\int_{0}^{t}\int_{0}^{J}G(t-s,x-y)\mf g(s,y,\mf u(s,y))dyds \\
&\quad +\int_{0}^{t}\int_{0}^{J}G(t-s,x-y)\sigma(s,y,\mf u(s,y))\,\mf W
(dyds)
\end{split}
\end{equation}
where $G:\R_+\times[0,J]\rightarrow\R$ is the fundamental solution of 
the heat equation 
\begin{align*}
\partial_t G(t,x) 
&=\frac{1}{2}\,\partial_x^2 G(t,x)  \\
G(0,x)&=\delta(x).
\end{align*}
where $[0,J]$ is the circle with endpoints identified.  
Furthermore, the final integral in \eqref{eq:weak-form} is a white 
noise integral in the sense of Walsh \cite{wals}. We give more 
information about the heat kernel in Section \ref{sec:heat-kernel}. 
Given an initial profile $\mf u_0$ that is continuous then it is well known
that there exists a unique strong solution to \eqref{eq:she} (see for example
section 6, page 23 of \cite{spde-utah}; the proof there can be easily
modified to cover \eqref{eq:she}). We are now ready to state the main results of this paper.

\begin{thm}
\label{th:small-ball} 
Consider the solution to \eqref{eq:she} and let the assumptions \eqref{eq:sigma:lip:2} and \eqref{eq:sigma:ue} hold. Then
\begin{enumerate}
 \item[(a)] There is a ${\mathscr D}_0  (J,\cl, \cu)>0 $ and positive constants 
$\mf{C}_0,\mf{C}_1,\mf{C}_2,\mf{C}_3$ depending only on
  $d, \cl, \cu$ and  $\varepsilon_0$ additionally  depending on $\sup\limits_{t, x, {\mf u} } \mid {\mf g}(t,x,{\mf u})\mid$
such that for any $\dlip<\dlip_0$ and all $0 < \varepsilon < \varepsilon_0,\,T>0$ we have
\begin{align} \label{corder}
\mf{C}_0\exp\left(-\mf{C}_1\frac{TJ}{\varepsilon^{6}}\right)
\leq P\left(\sup_{\substack{0\leq t\leq T\\x\in[0,J]}}|\mf u(t,x)|
 \leq\varepsilon\right)  \leq \mf{C}_2\exp\left(-\mf{C}_3\frac{TJ}{\varepsilon^{6}}\right).
\end{align}

\item[(b)] For any ${\mathscr D}$ and $0<\delta<1$, there exist positive 
constants $ \mf{C}_0,\mf{C}_1, \mf{C}_2,\mf{C}_3$  depending only on $d, \cl, \cu$ and
 $\varepsilon_0$ additionally depending on $J, \dlip,\delta, \sup\limits_{t, x, {\mf u} } \mid {\mf g}(t,x,{\mf u})\mid$ 
 such that for all $0 < \varepsilon < \varepsilon_0,\, T>0$ we have
  \begin{align}  \label{acorder}
\mf{C}_0\exp\left(-{\mf C}_1\frac{TJ^{1+(\delta/2)}}{\varepsilon^{6 + \delta}}\right)
\leq P\left(\sup_{\substack{0\leq t\leq T\\x\in[0,J]}}|\mf u(t,x)|
 \leq\varepsilon\right)  \leq \mf{C}_2\exp\left(-{\mf C}_3\frac{TJ}{(1+J\dlip^2)\varepsilon^{6}}\right).
\end{align}
\end{enumerate}
\end{thm}

As stated earlier in the introduction, in \cite{spde-utah}, page 168,
Theorem 5.1, Xiao proves a result similar to Theorem
\ref{th:small-ball} in the Gaussian case, including a term
$\varepsilon^{-6}$ in the exponent.  His argument builds on techniques
from Gaussian processes, in particular Robeva and Pitt \cite{rp04}.
Xiao's condition C3', which is related to the Gaussian concept of
local nondeterminism, is not always easy to verify.  His result
does not seem to carry over to \eqref{mainspde} even in the case 
where the coefficients are functions of $t,x$ but not of $\mf u$, and then the 
solution is a Gaussian process.  In contrast, we make key use of
the Markov property of \eqref{mainspde}, which allows us to extend our
results to the non-Gaussian case in which the equation has
coefficients that depend on the solution.

Before stating our next result we define the class
of predictable functions.

\begin{defn} Let $S$ be the set consisting of  functions $f: \R\times[0,J]\times \Omega \rightarrow \R^d$ of the form
  $$f(x,t,\omega) = X(\omega)\cdot 1_A(x)\cdot1_{(a,b]}(t) ,$$ with $ 0 < a < b <\infty$, $A \subset \R$, $X$ an ${\mathcal F}_a$ measurable random variable, and consider the {\it predictable} sigma-algebra $\mathcal P$ generated by all functions in $S$. A function $\mf h(t,x,\omega):\R_+\times\R\times\Omega\rightarrow\R^d$ is said to be predictable if it is measurable with respect to ${\mathcal P}$. We will say a predictable function $\mf h \in {\mathcal P}{\mf C}_b^2$, if with probability one $\mf h$, $\partial_t\mf h$, and  $\partial_x^2\mf h$  are uniformly bounded by a deterministic constant $\mf H$.  
  \end{defn}

\begin{thm}
\label{th:support} 
Consider the solution to \eqref{eq:she} and let the assumptions \eqref{eq:sigma:lip:2} and \eqref{eq:sigma:ue} hold. Let $\mf u_0, \mf h \in {\mathcal P}{\mf C}^2_b$ and assume 
\begin{equation*}
\sup_{x\in[0,J]}|\mf u_0(x)-\mf h(0,x)|<\varepsilon/2
\end{equation*}
almost surely.  Then 
\begin{enumerate}
 \item[(a)] There is a ${\mathscr D}_0  (J,\cl, \cu)>0 $ and positive constants 
$ \mf{C}_0,\mf{C}_1,\mf{C}_2,\mf{C}_3$ depending only on
$d, \cl, \cu$  and $\varepsilon_0$ additionally  depending on ${\mf H}, \sup\limits_{t, x, {\mf u} } \mid {\mf g}(t,x,{\mf u})\mid$
such that for any $\dlip<\dlip_0$ and all $0 < \varepsilon < \varepsilon_0,\,T>0$ we have
\begin{equation} \label{corder-1}
\mf{C}_0\exp\left(-\mf{C}_1\frac{TJ}{\varepsilon^{6}}\right)
\leq P\left(\sup_{\substack{0\leq t\leq T\\x\in[0,J]}}|\mf u(t,x)-\mf h(t,x)|
 \leq\varepsilon\right)  \leq \mf{C}_2\exp\left(-\mf{C}_3\frac{TJ}{\varepsilon^{6}}\right).
\end{equation}

\item[(b)] For any ${\mathscr D}$ and $0<\delta<1$, there exist positive constants  
$ \mf{C}_0,\mf{C}_1, \mf{C}_2,\mf{C}_3$  depending only on $d, \cl, \cu$ and
 $\varepsilon_0$ additionally depending on $J, \dlip,\delta, {\mf H}, \sup\limits_{t, x, {\mf u} } \mid {\mf g}(t,x,{\mf u})\mid$  such that for all $0 < \varepsilon < \varepsilon_0,\, T>0$ we have
\begin{equation}  \label{acorder-1}
\begin{split}
\mf{C}_0\exp\left(-\frac{{\mf C}_1TJ^{1+(\delta/2)}}{\varepsilon^{6 + \delta}}\right)
&\leq P\left(\sup_{\substack{0\leq t\leq T\\x\in[0,J]}}|\mf u(t,x)-\mf h(t,x)|  
   \leq\varepsilon\right)   \\
&\leq \mf{C}_2\exp\left(-\frac{{\mf C}_3TJ}{(1+J\dlip^2)\varepsilon^{6}}\right).
\end{split}
\end{equation}
\end{enumerate}
\end{thm}
Our result is similar to the support theorem for Brownian motion given in \cite[Proposition 6.5 and Theorem 6.6 on pages 59-60]{B95}. Theorem \ref{th:support} says for any {\it nice} function $\mf h$ there is a positive probability that the solution $\mf u$ gets {\it close} (within $\varepsilon$) to $\mf h$ provided that the initial profile ${\mf u}_0$ is close (within $\varepsilon/2$) to $\mf h(0,\cdot)$. Support theorems, of a slightly different flavour, for \eqref{eq:she} have been studied in the literature.
In \cite{bms95}, a support theorem is proven for \eqref{eq:weak-form} when $d=1$ and $J=1$.  Suppose $S(h)$ denotes the solution \eqref{eq:weak-form} when the white noise $\dot{ W}$ is replaced by $\dot { h} \in \mathcal{H}$, where $$ \mathcal{H} = \{ {h}:[0,T] \times[0,1] \rightarrow \R : {h} \mbox{ is absolutely continuous and }  \dot{h} \in L^2([0,T]\times [0,1]) \}.$$ When $u(0,x)$ is H\"older $\alpha$ with $\alpha < \frac{1}{2}$,  $g$ has bounded derivatives up to order $3$,  and $\sigma$ is a Lipschitz function they show that the support of ${\mathbb P}\circ u^{-1}$ is the closure in the H\"older topology  of the set $\{S(h): h \in \mathcal{H}\}$. 

We will now make a couple of remarks. These could be of independent interest.

\begin{rem}
  \begin{itemize}
    \item[(a)] For Theorem \ref{th:small-ball}, our assumptions on $\sigma$, $\mf g$ need only hold until $\mf u$ exits from the $\varepsilon$-ball and respectively until $\mf u$ exits from the $\varepsilon$-ball around $\mf h$ for Theorem \ref{th:support}.
  \item[(b)]  It will be clear from our proofs of part (a) in Theorem
  \ref{th:small-ball} and Theorem \ref{th:support} that in fact
  $\dlip_0(J,\cl,\cu)= D_0J^{-\frac12}$ where $D_0$ depends on $\cl,
  \cu$ only. For part (b) of Theorem \ref{th:small-ball} and Theorem
  \ref{th:support} we can choose $\varepsilon_0= e_0\cdot
  (J\mathscr{D}^2)^{-\frac{2}{\delta}}$ where $e_0$ depends only on
  $J,\cl,\cu, \sup\limits_{t, x, {\mf u} } \mid {\mf g}(t,x,{\mf
    u})\mid,$ and ${\mf H}$.
\item[(c)] Theorem \ref{th:small-ball} shows that under certain conditions. $\epsilon^6 \log P\left(\sup_{x\in [0,J], \, t\le T} |\uv(t,x)|\le \epsilon\right)$ is bounded away from $0$ and $\infty$ as $\epsilon\to 0$.  An important question is whether the limit exists (the limit is then called the {\it small ball constant}). This is in general a very difficult question (see the discussion in in Section 6 of \cite{LS01}).
  \end{itemize}
\end{rem}

As noted earlier in the introduction, sharp small ball estimates
for Gaussian processes are obtained using their special structures.
Gaussian processes have many detailed properties and these do
not hold for the stochastic heat equation \eqref{eq:she}. Therefore
the proofs of our above results will not follow by translating
techniques used in proving small-ball probabilities in the literature.

However, by {\it freezing} the coefficient ${\mf \sigma}(t,x, {\mf u})$
we may approximate $u$ by a Gaussian random field, at least in a small
time region.  One of the key arguments of the paper is in showing that
the error in the approximation can be well controlled if the time
interval where the coefficient is frozen is suitably chosen.

The stochastic heat equation also has a Markov property with respect
to the time parameter $t$, and this property plays an essential role
in the proof of Theorem \ref{th:small-ball}. Thanks to this property,
we are able to reduce our analysis to the behaviour of the solution in
small time intervals of order $\varepsilon^4$.  Roughly speaking, we
show that the probability that the solution remains within
$\varepsilon$ in this small time increment is like
$\exp(-C\varepsilon^{-2})$; this is the content of Proposition
\ref{prop:bd}. Since there are $O(\varepsilon^{-4})$ such time
intervals, we get our result.

  The rest of the paper is organized as follows. In Section
  \ref{sec:opmr} we state the key Proposition \ref{prop:bd} required
  to prove our main results, and reduce our problem to the case
  $\g\equiv \mf 0$ and $J=1$ using a couple of lemmas. We also explain
  how Theorem \ref{th:support} follows from Theorem
  \ref{th:small-ball}. Then in Section \ref{sec:prelim} we provide
  some heat kernel estimates which yield critical tail bounds on the
  noise term in Lemma \ref{lem:Ntail}. We conclude the paper with
  Section \ref{sec:prop} where we prove Proposition \ref{prop:bd} and
  then Theorem \ref{th:small-ball}.

{\em Convention on constants:} Throughout the paper C denotes a
positive constant whose value may change from line to line. All other
constants will be denoted by $C_1,C_2, \ldots$. These are positive
with their precise values being not important. The dependence of
constants on parameters when relevant will be denoted by special
symbols or by mentioning the parameters in brackets, for e.g.
${\mathscr C}, C_1(\sigma,J)$.

\section{Some reductions and the key proposition}\label{sec:opmr}

We first explain how Theorem \ref{th:support} follows immediately from Theorem \ref{th:small-ball}. We next discuss how the analysis of Theorem \ref{th:small-ball} can be reduced to the case $\g\equiv \mf 0$ and $J=1$.  Finally we state the key Proposition \ref{prop:bd} which is the main ingredient in the proof of Theorem \ref{th:small-ball} and whose proof will occupy the majority of the paper.

\subsection{Proof of Theorem \ref{th:support}} 
Before we proceed to the proof of Theorem \ref{th:small-ball}, let us discuss how Theorem \ref{th:support} follows from Theorem \ref{th:small-ball}. Let
\begin{equation*}
H=\frac12\,\partial_x^2-\partial_t
\end{equation*}
be the heat operator on $(t,x)\in\R_+\times[0,J]$ where as usual, 
$[0,J]$ is the circle with endpoints identified.  Let
$\mf h_0(x)=\mf h(0,x)$ and let
\begin{align*}
\mf w(t,x)&=\mf u(t,x)-\mf u_0(x)-\mf h(t,x)+\mf h_0(x)  \\
\mf g_1(t,x, \mf w)&=-{\mf g}(t,x, \mf u)-H\mf u_0(x)-H\mf h(t,x)+H\mf h_0(x)  \\
\sigma_1(t,x,\mf w)& =\sigma(t,x,\mf u).
\end{align*}
We see that $\sigma_1$ is Lipschitz in $\mf w$ with the same 
Lipschitz constant as $\sigma$.  Furthermore, by the assumptions of 
Theorem \ref{th:support}, $\mf g_1$ is uniformly bounded by a 
deterministic constant, almost surely.  

Then $\mf w$ satisfies
\begin{align*}
\partial_t\mf w(t,x)
&=\frac12\,\partial_x^2\mf w(t,x)-\mf g_1(t,x, \mf w)
 +\sigma_1(t,x,\mf w)\,\dot{\W}(t,x),  \\
\mf w(0,x)&=\mf 0.
\end{align*}
  
Now $\sup_x|\uv_0(x)-{\mf h}_0(x)|\le \varepsilon/2$, and so $\sup_x|{\mf w}(t,x)|\le \varepsilon/2$ implies $\sup_x|\uv(t,x)-{\mf h}(t,x)|\le \varepsilon$. 

Thus Theorem \ref{th:support} follows from applying Theorem \ref{th:small-ball} to $\mf w$. 
\qed

The rest of the paper will be focused on the proof of Theorem \ref{th:small-ball}. 

\subsection{Reduction to the case $\g\equiv\mf 0$ }  We show that it is enough to prove Theorem \ref{th:small-ball}  when  $\mf g\equiv\mf 0$. We will  need the  following Girsanov lemma and moment estimate on the Radon-Nikodym derivative.
\begin{lem}   \label{lem:girsanov}
Let $\mf f:[0,\infty) \times \R^d\times \Omega \to \R^d$ be a predictable function which is uniformly bounded by $M>0$. Let $P_t$ be $P$ restricted to $\mathcal{F}_t$  and consider the measure $Q_t$ given by 
\begin{equation} 
\label{eq:girsanov}
\frac{dQ_t}{dP_t}=\exp\left(Z_t^{(1)}- \frac12 Z_t^{(2)}\right),
\end{equation}
where
\begin{align*}
 Z_t^{(1)}&:=\int_0^t \int_0^J \mf f(s,x) \cdot   \mf W(dxd s),\\
 Z_t^{(2)}&:= \int_0^t\int_0^J |\mf f(s,x)|^2 \,dxds, 
\end{align*}
and $\mf f\cdot \mf W$ is the dot product of $\mf f$ and $\mf W$ in $\R^d$.
Then
\begin{itemize}
  \item[(a)] Under the measure $Q_t$,
\[ \dot{\widetilde{\mf W}} (s,x) := \dot{\mf W}(s,x) - \mf f(s,x),\; x\in [0, J],\, s\in [0,t] \]
is a $d$ dimensional vector of i.i.d. two-parameter white noise.

\item[(b)] Furthermore 
 \begin{equation} \label{eq:estimate-girsanov}
1 \leq E\left[\left( \frac{dQ_t}{dP_t} \right)^2\right] \leq \exp\left(M^2tJ\right).
\end{equation}
  
\end{itemize}
\end{lem}
\begin{proof}  (a) A stronger version of the statement can be found in \cite{allo} but this is enough for our purposes. While \cite{allo} deals only with $d=1$, the extension below to higher dimensions is immediate. 

(b) Since $dQ_t/dP_t$ is a Radon-Nikodym derivative,
\begin{equation*}
1=E\left[\frac{dQ_t}{dP_t}\right]
 =E\left[\exp\left(Z_t^{(1)}-\frac{1}{2}Z_t^{(2)}\right)\right]
\end{equation*}
and replacing $\mf f$ by $2\mf f$ in the definitions of 
$Z_t^{(1)},Z_t^{(2)}$, we get
\begin{equation}
\label{eq:one-equals}
1=E\left[\exp\left(2Z_t^{(1)}-2Z_t^{(2)}\right)\right].
\end{equation}
Next, note that $0\leq Z^{(2)}_t\leq M^2tJ$ and therefore
\begin{equation}
\label{eq:bounds-exp-Z-2}
1\leq\exp\left(Z^{(2)}_t\right)\leq\exp\left(M^2tJ\right).
\end{equation}
Combining \eqref{eq:one-equals} and \eqref{eq:bounds-exp-Z-2}, we get 
\begin{align*}
E\left[\left(\frac{dQ_t}{dP_t}\right)^2\right]
&=E\left[\exp\left(2Z_t^{(1)}-Z_t^{(2)}\right)\right]  \\
&= E\left[\exp\left(2Z_t^{(1)}-2Z_t^{(2)}\right)
 \cdot\exp\left(Z_t^{(2)}\right)\right], 
\end{align*}
and we obtain \eqref{eq:estimate-girsanov} from \eqref{eq:one-equals} and \eqref{eq:bounds-exp-Z-2}.
\end{proof}

Using the above lemma we now explain how it is enough to prove 
Theorem \ref{th:small-ball} when  $\mf g\equiv\mf 0$. Consider the 
event 
\begin{equation*}
A=\left\{\sup_{s\le T, \, y\in [0,J]} |\uv(s,y)|\le \varepsilon\right\}.
\end{equation*}
Consider \eqref{eq:she} with ${\mf g}\equiv \mf 0$, and write
\bes
\begin{split} 
\partial_t\uv(t,x)
&=\frac12\, \partial_x^2\uv(t,x)
  + \sigma\big(t,x,\uv(t,x)\big)\,\dot{\W}(t,x) \\
&= \frac12\,\partial_x^2\uv(t,x) + {\mf g}\big(t,x, \uv(t,x)\big) \\
&\qquad \qquad\; \;\;\quad+ \sigma\big(t,x,\uv(t,x)\big)\, \left(\dot{\W}(t,x) -\sigma^{-1}\big(t,x,\uv(t,x)\big) {\mf g}(t,x)\right).
\end{split}
\ees
Let 
\[ {\mf f}(t,x) = \sigma^{-1}\big(t,x,\uv(t,x)\big) {\mf g}(t,x),\]
and note that $\mf f$ is uniformly bounded by some $M>0$ by the assumptions on $\g$ and \eqref{eq:sigma:ue}.  Define $Q_T$ as in \eqref{eq:girsanov}. From Lemma \ref{lem:girsanov}, we know that $\dot{\widetilde \W}(s,x):= \dot \W(s,x) -{\mf f}(s,x),\; x\in [0,J],\, s\in [0,T]$ is a white noise under $Q_T$. Therefore the distribution of $\uv$ under $Q_T$ corresponds to the case when $\g$ is present in \eqref{eq:she}. Now 
\be\label{eq:P:Q}
\begin{split}
Q_T(A)&=E\left[\mathbf 1_A \frac{dQ_T}{dP_T}\right] \\
&\le \sqrt{P(A)}\cdot  \sqrt{E\left[\left(\frac{dQ_T}{dP_T}\right)^2\right]}\\
&\le \sqrt{P(A)} \cdot \exp\left(\frac{M^2TJ}{2}\right).
\end{split}
\ee
This explains how we get a similar upper bound for nonzero $\g$ as when $\g \equiv\mf  0$ with different 
constants. For the lower bound consider instead \eqref{eq:she} with nonzero $\g$. We can write 
\begin{equation*}
\partial_t\uv(t,x)
=\frac12\,\partial_x^2\uv(t,x)
   + \sigma\big(t,x,\uv(t,x)\big)\,\left(\dot{\W}(t,x) + \sigma^{-1}\big(t,x, \uv(t,x)\big)\g(t,x)\right).
\end{equation*}
Now let $\mf f(t,x) = -\sigma^{-1}\big(t,x,\uv(t,x)\big) {\mf g}(t,x)$ and define $Q_T$ as before. Note that in this case $\dot{\widetilde \W}= \dot \W+{\mf f}$ is a white noise under $Q_T$, and so $\uv$ has the distribution of \eqref{eq:she} with $\g \equiv\mf  0$. Follow the same argument as in \eqref{eq:P:Q} to obtain a similar lower bound for nonzero $\g$ as when $\g\equiv\mf 0$. 

\subsection{Reduction to the case $J=1$} \label{jeqt1} Due to the previous subsection we can now assume $\g\equiv \mf 0$. Let $\mf u(t,x)$ be the solution to \eqref{eq:she} with $\g\equiv 0$. We now reduce to the case $J=1$. Consider the  function 
\begin{equation}
\label{eq:u-v-scaling}
\mf v(t, z) =  J^{-1/2} \mf u(J^2 t , J z),\quad t\ge 0, \, z\in [0,1],
\end{equation}
with initial profile $\mf v_0(z) = J^{-1/2} \mf u_0(Jz),\, z\in [0,1]$.

Let us denote the heat kernel by 
$G^{(J)}(t,x)$ to emphasize the dependence on $J$. The scaling 
relation below is immediate.  
\begin{equation}
\label{eq:kernel:scaling} 
G^{(1)}(J^{-2}t, J^{-1} x) = J\cdot G^{(J)}(t, x),\quad x\in [0,J],\, 
t\ge 0.  
\end{equation}
The following distributional identity for white noise is 
well known.
\begin{equation}
\label{eq:white:scaling} 
{\mf W}^{J^{2}, J}(dy\, ds) 
:= J^{-3/2}{\mf W} ( J dy \, J^{2} ds) \stackrel{\mathcal D}{=} {\mf 
W} (dy\, ds), 
\end{equation}
where $\stackrel{\mathcal D}{=}$ denotes equality 
in distribution.

\begin{lem} \label{lem:scaling} The random field $\mf v(t,x): t\geq0, x\in[0,1]$ 
is the mild solution to the stochastic heat equation on $[0,1]$ with white noise 
$\dot {\mf W}^{J^2, J}$:
\begin{equation}\label{eq:v}
\begin{split}
\partial_t \mf v(t,x)&=\frac{1}{2}\,\partial_x^2\,\mf v(t,x)+\sigma^{
(J)}\left(t, x, \mf v(t,x)\right)\cdot\dot{\mf W}^{J^2, J}(t,x) \\
\mf v(0,x)&=\mf v_0(x),  
\end{split}
\end{equation}
where 
\begin{equation*}
\sigma^{(J)}(t,x, \mf u) := \sigma (J^2t, J x, J^{1/2} \mf u).
\end{equation*}
\end{lem}
\begin{proof}
From \eqref{eq:weak-form} one obtains
\begin{align*} \mf v(t,z)&=J^{-1/2}\int_{0}^{J}G^{(J)}(J^2t,Jz-y)\,\mf u_0(y)\,dy  \\
&\quad +J^{-1/2}\int_{0}^{J^2t}\int_{0}^{J}G^{(J)}(J^2t-s,Jz-y)\,\sigma\left(s, y, 
\mf u(s,y)\right)\,\mf W(dyds)\\
&= J^{-3/2}\int_{0}^{J}G^{(1)}(t,z-J^{-1}y)\,\mf u_0(y)\,dy  \\
&\quad +J^{-3/2}\int_{0}^{J^2t}\int_{0}^{J}G^{(1)}(t-J^{-2}s,z-J^{-1}y)\,\sigma\left(s, y, 
\mf u(s,y)\right)\,\mf W(dyds)\\
&= \int_{0}^{1}G^{(1)}(t,z-w)\,\mf v_0(w)\,dw  \\
&\quad +\int_{0}^{t}\int_{0}^{1}G^{(1)}(t-r,z-w)\,\sigma\left(J^2 r, J w, 
J^{1/2}\mf v(r,w)\right)\,\mf W^{J^2, J}(dwdr),
\end{align*}
where we have used \eqref{eq:kernel:scaling} for the second equality, and \eqref{eq:white:scaling} for the last equality. 
\end{proof}

The reduction to $J=1$ then follows easily using Lemma \ref{lem:scaling}.  Indeed, write $s=J^2t$, $y=Jx$ and note that
\begin{align*}
P\left(\sup_{\substack{0\leq s\leq T\\y\in[0,J]}}|\mf u(s,y)|
 \leq\varepsilon\right)
&=P\left(\sup_{\substack{0\leq J^2t\leq T\\Jx\in[0,J]}}J^{-1/2}|\mf u(J^2t,Jx)|
 \leq\varepsilon J^{-1/2}\right)   \\
&=P\left(\sup_{\substack{0\leq t\leq TJ^{-2}\\x\in[0,1]}}|\mf v(t,x)|
 \leq\varepsilon J^{-1/2}\right).
\end{align*}
Assuming we have Theorem \ref{th:small-ball} for $J=1$ we will obtain the result for general $J$ from the above. 
\begin{rem}[{\bf Important}] Note that the function $\sigma^{(J)}$ satisfies \eqref{eq:sigma:ue} with the same $\cl, \cu$. However the Lipschitz constant for $\sigma^{(J)}$ is $J^{1/2}\dlip$, and \textnormal{not} $\dlip$. This is the reason for the somewhat strange expressions for the upper bounds in \eqref{acorder} and \eqref{acorder-1}.
\end{rem}

\begin{rem} Thanks to the above reductions, we will assume for the rest of the article that $J=1$ and $\g\equiv \mf 0$.
\end{rem}

\subsection{Key proposition}
We divide the time interval $[0,T]$ into increments of length
$c_0\varepsilon^4$ where $c_0=c_0(\cl, \cu)$ is chosen so that 
\begin{equation} \label{chc0} 
 0<c_0 < \max\left\{\left(\frac{\mf K_2}{36 \log ({\mf K_1}) {\mathscr C}_2^2}\right)^2,1\right\}.
\end{equation}
Above ${\mf K}_1$ and ${\mf K}_2$ are universal constants specified in the statement of Lemma \ref{lem:Ntail}. Consider time points \[t_n= n c_0\varepsilon^4,\quad n\ge 0,\] and  let $I_n:=[t_n, t_{n+1}]$ be the $n$th time increment. Let 
\[ n_1:=\min\{n \ge 1: t_n > T\} \]
be the smallest $n$ for which the time interval $I_n$ is completely outside $[0,T]$.

We shall similarly consider a discrete set of spatial points separated by $c_1\varepsilon^2$, where $c_1^2=\theta c_0$ with $\theta=\theta(\cl, \cu)>0$ is chosen so that 
\begin{equation}
\label{chth}
\begin{split}
& \theta \ge \max\left\{2, 4\log\left(\frac{1}{2c_0}\right)\right\}\quad   \text{ and } \quad \frac{C_{10}}{C_8} \sum_{k\ge 1} \exp\left(-\frac{\theta k^2}{8}\right) <\frac16.
\end{split}
\end{equation}
The constants $C_8$ and $C_{10}$ are specified in the statement of Lemma \ref{lem:N:var} and depend on $\cl$ and $\cu$ only. Consider discrete spatial points
\[ x_n=n c_1\varepsilon^2,\quad  n\ge 0,\]
and let $J_n:=[x_n,\, x_{n+1}]$ be the $n$th space increment. Let 
\begin{equation*}
n_2:=\min\{n \ge 1: x_n >1\} 
\end{equation*}
be the smallest $n$ for which the space interval $J_n$ is completely outside $[0,1]$.  Note that 
\begin{equation}
\label{eq:est-n2}
n_2\leq (c_1\varepsilon^2)^{-1}+1.
\end{equation}

We will first define a sequence of sets which we can use to  prove the lower bounds in \eqref{corder} and \eqref{acorder}. Let $A_{-1}=\Omega$ and for $n\ge 0$ define events
  \be \label{eq:An} A_n =\left\{ |\uv(t_{n+1}, x)|\le \frac\varepsilon 3 \;\;\forall x, \text{ and } |\uv(t,x)|\le \varepsilon \;\;\forall t\in I_n, \, x\in [0,1] \right\}.\ee
  Thus $A_n$ is the event that $\uv(t,\cdot)$ is everywhere of modulus at most $\varepsilon$ in the time interval $I_n$, and is everywhere of modulus at most $\varepsilon/3$ at the terminal time $t_{n+1}$.
We will next define a sequence of sets which we use to  prove the upper bounds in \eqref{corder} and \eqref{acorder}.  Denote by
\begin{equation*} p_{ij} = (t_i, x_j), 
\end{equation*}
the left hand corner of the box $I_i \times J_j$. %
 Let $F_{-1}= \Omega$ and for $n\ge 0$, define 
\[ F_n = \left \{\vert \uv(p_{nj})\vert \le \varepsilon \text{ for all } j \le n_2-2 \right\}.\] 
By the above constructions of $A_n$ and $F_n$, the proposition below along with the {\it Markov property} will be the key step in the proof Theorem \ref{th:small-ball}.

\begin{prop} \label{prop:bd} Fix $\g\equiv {\mf 0},\, J=1$. Consider the solution to \eqref{eq:she} with $\mf u_0(x)\equiv 0$ and let the assumptions \eqref{eq:sigma:lip:2} and \eqref{eq:sigma:ue} hold. Then
\begin{enumerate}
  \item[(a)] For all $\mathscr D>0$, there exist constants $\varepsilon_0(\cl,\cu,\dlip)>0$ and $C_4,C_5>0$ depending only $\cl, \cu$ such that for $0<\varepsilon <\varepsilon_0$ and $n\ge 1$
\begin{align} \label{ubd}
 P\Big(F_n \Big\vert \bigcap_{k=0}^{n-1} F_k\Big) \le C_4 \exp\left(- \frac{C_5}{(1+\dlip^2)\varepsilon^{2}}\right).
\end{align}
\item[(b)] There is a ${\mathscr D}_0(\cl,\cu) >0 $ and constants $\varepsilon_0, C_6, C_7$ depending only on $\cl,\cu$ such that for any  ${\mathscr D} < {\mathscr D}_0$ , $0<\varepsilon<\varepsilon_0$ and $n\ge 0$
  \begin{align} \label{lbdc}
    P\Big(A_n \Big\vert \bigcap_{k=-1}^{n-1} A_k \Big) \ge C_6 \exp \left(- C_7\varepsilon^{-2}\right).
  \end{align}
    \end{enumerate}
 \end{prop}
The majority of the work in the paper will be to prove the above Proposition. In the last section we will explain how Theorem \ref{th:small-ball} follows from this.

\section{Preliminaries} \label{sec:prelim}

In this section we state and prove some preliminaries with regard to
(\ref{eq:she}) that we will need for the proof of Proposition \ref{prop:bd}. Recall that we are restricting ourselves to $J=1$ and $\g\equiv \mf 0$. 

\subsection{Heat Kernel Estimates}
\label{sec:heat-kernel}
In this subsection we prove a few preliminary results about
the heat kernel $G(t,x)$ which was mentioned in the introduction. 
$G$ is given by
\begin{equation*}
 G(t,x) = \sum_{n\in \Z} (2\pi t)^{-1/2} \exp \left(- \frac{(x+n)^2}{2t}\right).
\end{equation*}

The following lemma is well-known for $x\in\R$, see for example 
\cite{spde-utah}, Lemma 4.3,
page 126. We give a brief proof for the case $x\in[0,1]$ (the 
circle). 
 
\begin{lem}\label{lem:g}There exists a constant $C_0 >0$ such that for all $0<s<t\le 1$ with $|t-s|\le 1$ and $x, y \in [0,1]$, we have
\begin{align}
\label{eq:g1}\int_0^t \int_0^1 \left[G(s, x-z) - G(s, y-z)  \right]^2 \, dz \, ds  & \le C_0|x-y|, \\
\label{eq:g2}\int_s^t \int_0^1 G^2(t-r, z) \, dz dr & \le C_0|t-s|^{1/2}, \\
\label{eq:g3} \int_0^s \int_0^1 \left[ G(t-r, z) - G(s-r, z)\right]^2\, dz dr & \le C_0|t-s|^{1/2}.
\end{align}
\end{lem}
\begin{proof} From the standard Fourier decomposition  we have 
\[ G(t,x) = \sum_{k\in \Z} \exp\Big( -\frac{(2\pi k)^2  t}{2} \Big)\cdot \exp\big(i (2\pi k) x\big).\] 
Using the orthogonality of $\{\exp\big(i(2\pi k)z\big): k \in \Z\}$ in $L^2([0,1])$, it is immediate that there is a $c_1 >0$ such that
\begin{align*}
\int_0^t \int_0^1 &\left[G(s, x-z) - G(s, y-z)  \right]^2 \, dz \, ds  \\ &= C\int_0^t ds \, \sum_{k \ge 1} \exp\big(-(2\pi k)^2 s\big)\cdot \big\vert 1- \exp\big(i(2\pi k) (x-y)\big) \big\vert^2 \\
  &\leq C\sum_{k \ge 1} \frac{1}{k^2} \cdot \Big|1\wedge (2\pi k) |x-y|\Big|^2 \\\
 &\le C|x-y|.\end{align*}
The second inequality  above is  obtained by using Fubini's 
theorem, integrating over $s$ and using $|1-e^{ix}| \le 2\wedge |x|.$ 
The last inequality above is obtained by splitting the sum according 
to whether $k$ is less than or greater than $(2\pi|x-y|)^{-1}$. Thus 
we have obtained \eqref{eq:g1}.

As for \eqref{eq:g2} we integrate over $z$ first and using
orthogonality again we obtain 
\begin{align*}
  \int_s^t \int_0^1 &G^2(t-r, z) \, dz dr \\
  &\leq C \int_0^{t-s} dr \, \sum_{k\ge 0} \exp\big(- (2\pi k)^2 r\big) \\
&\le  C(t-s) + C\sum_{k\ge 1} \frac{1}{(2\pi k)^2} \Big[1- \exp\big(-(2\pi k)^2 (t-s)\big)\Big] \\
  & \le C(t-s) + C\sum_{k\ge 1} \frac{1}{(2\pi k)^2 }\Big[1\wedge (2\pi k)^2 (t-s) \Big]\\
  &\le C(t-s)^{1/2}. 
\end{align*} 
We obtain the last inequality above by splitting the earlier sum according to whether $(2\pi k)^2 (t-s)$ is less than or greater than $1$. Thus we have obtained \eqref{eq:g2}.

 For \eqref{eq:g3}, using orthogonal basis as above we have that
 \begin{align*}
   \int_0^s \int_0^1 &\left[ G(t-r, z) - G(s-r, z)\right]^2\, dz dr\\
   &= \int_0^s dr \, \sum_{k \ge 1} \exp\big(-(2\pi k)^2 r \big)\cdot \Big[1- \exp\big(-(2\pi k)^2 (t-s)\big)\Big].
   \end{align*}
Using Fubini and integrating each of the terms over $r$, we obtain \eqref{eq:g3}.
\end{proof}

For $x\in [0,1]$, the circle with end points identified, define 
\begin{equation} \label{eq:xstar} x_* = \begin{cases}x,\, & 0\le x\le \frac 12, \\
x-1,\, & \frac12<x\le 1. \end{cases}
\end{equation} 
We will need the following  comparison between the heat kernel on the circle $[0,1]$ with the heat kernel on $\R$. 
\begin{lem} \label{lem:G:bd} There is a $C_G>0$ such that for all $t\le 1$ 
\[ G(t,x) \le \frac{C_G}{\sqrt{2\pi t}} \exp\Big(- \frac{x_*^2}{2t} \Big),\quad  x\in [0,1].\]
\end{lem}
\begin{proof} As we are working on the circle $[0,1]$ we have $G(t,x)= G(t,x_*)$. It is immediate to observe that $|x_*|\le \min(|x_*|, |x_*-1|, |x_*+1|)$, and  
\[ -(x_*+k)^2\le -x_*^2 - \frac{k^2}{2} \quad \text{ if } |k|\ge 2.\] 
Therefore, 
\begin{align*}
G(t,x_*) & \le \frac{3}{\sqrt{2\pi t}} \exp\Big(- \frac{x_*^2}{2t} \Big) + \sum_{|k|\ge 2} \frac{1}{\sqrt{2\pi t}} \exp\Big(- \frac{x_*^2}{2t} - \frac{k^2}{4t} \Big)\\
& \le \frac{3}{\sqrt{2\pi t}} \exp\Big(- \frac{x_*^2}{2t} \Big) +  \frac{1}{\sqrt{2\pi t}} \exp\Big(- \frac{x_*^2}{2t}\Big) \sum_{|k|\ge 2}e^{-k^2/4},
\end{align*}
uniformly for $t\le 1$. This completes the proof.
\end{proof}

\subsection{Noise Term Estimates}

Recall that  any solution ${\mf u}$ to \eqref{eq:she} with $J=1,\g\equiv \mf 0$ satisfies  \begin{align}
\mf u(t,x)&=\int_{0}^{1}G(t,x-y)\,\mf u_0(y)\,dy  +\int_{0}^{t}\int_{0}^{1}G(t-s,x-y)\,\sigma\left(s, y, \mf u(s,y)\right)\,\mf W(dyds),\label{eq:mild}
\end{align}
where $x-y$ denotes subtraction modulo $1$. The above is known as the 
mild form of the solution. We refer the reader to \cite{wals} or 
\cite{spde-utah} for a discussion of the stochastic integral with 
respect to white noise, and a treatment of the mild form. The 
Lipschitz assumption on $\sigma$ guarantees that a unique strong 
solution exists \cite[Theorem 6.4, page 26]{spde-utah}.

We shall denote the second term of \eqref{eq:mild}, i.e. {\it noise term}, by
\begin{equation}
\label{eq:noise}
\N(t,x) : = \int_0^t \int_0^1G(t-s, x-y)\sigma\left(s, y, \mf u(s,y)\right)
\,\mf W(dyds).
\end{equation}

\begin{lem} \label{lem:N:holder} There exist constants $C_1, C_2$ depending on the dimension $d$ such that for all $0<s<t<1$, $x, y\in [0,1]$ and $\lambda>0$, 
\begin{align*}
P\left( \big|\N(t,x)- \N(t,y)\big|>\lambda\right)& \le C_1\exp\left(- \frac{C_2\lambda^2}{\mathscr C_2^2|x-y|}\right) \\
P\left( \big|\N(t,x)- \N(s,x)\big|>\lambda\right)& \le C_1\exp\left(- \frac{C_2\lambda^2}{\mathscr C_2^2|t-s|^{1/2}}\right)
\end{align*}
\end{lem}
\begin{proof}
One can use Lemma \ref{lem:g} and follow the argument in Corollary 4.5 on page 127 of \cite{spde-utah} to obtain the result. We sketch the details. First note that it is enough to prove the above inequalities for each of the components of $\N$. Let us focus on the first coordinate $N_1$. For the first inequality one observes that for $s\in[0,t]$
\begin{equation*}
\int_0^s \int_0^1 \left[G(t-r, x-z)- G(t-r, y-z)\right]\, \sigma\left(r, z, \mf u(r,z)\right)\,\mf W(dzdr) 
\end{equation*}
is an $\mathcal{F}_t$-martingale whose value at time $t$ is $\N(t,x)- \N(t,y)$. Thus 
$N_1$ is also a martingale. Any one dimensional martingale is a 
time-changed Brownian motion and \eqref{eq:g1} gives a bound of 
$C\mathscr C_2^2|x-y|$ on the time change. One then uses the 
reflection principle to get the bound on the probability. Next 
consider the martingale
\begin{equation*}
\mf M_q= \int_0^q \int_0^1 G(t-r, x-z)\, 
 \sigma\left(r, z, \mf u(r,z)\right)\,\mf W(dzdr) -\N(s,x),\quad s\leq q\leq t.
\end{equation*}
The second bound can be proved similarly using \eqref{eq:g2} and \eqref{eq:g3} to get a bound of $C\mathscr C_2^2|t-s|^{1/2}$ on the time change.  
\end{proof}

\begin{lem} \label{lem:Ntail}There exist constants $\mf K_1, \, \mf K_2>0$ depending on the dimension $d$ such that for all $\alpha, \lambda, \varepsilon>0$ we have
\begin{align}\label{eq:N:tail}
P \left(\sup_{\stackrel{0\le t\le \alpha\varepsilon^4}{x\in [0,\varepsilon^2]}} |\N(t,x)|>\lambda \varepsilon\right) \le  \frac{\mf K_1}{1\wedge \sqrt \alpha}\exp\left( - \mf K_2 \frac{\lambda^2}{\mathscr C_2^2\sqrt\alpha}\right).
\end{align}
\end{lem}
\begin{proof}  Let us first consider the case $\alpha\ge 1$. For $n\ge 0$, define the grid 
\[ \mathbb G_n=\left\{ \left(\frac{j}{2^{2n}}, \frac{k}{2^n}\right): 0\le j \le \alpha \varepsilon^42^{2n},\, 0\le k \le \varepsilon^2 2^n\right\}. \]
Let 
\[ n_0 = \lceil\log_2\big(\alpha^{-1/2}\varepsilon^{-2}\big)\rceil . \]
For $n< n_0$, the grid $\mathbb G_n$ consists simply of the point $(0,0)$. For $n\ge n_0$, the grid $\mathbb G_n$ has at most $4\alpha \varepsilon^4 2^{2n} \cdot \varepsilon^2 2^n \le 4 \cdot 2^{3(n-n_0)}$ many points. Fix 
\[ K= \left( 4 \sum_{n\ge 0} 2^{-n/4}\right)^{-1}.\]
Consider the events 
\[ A(n,\lambda) = \left\{\left\vert \N(p)- \N(q)\right\vert \le \lambda K \varepsilon 2^{-(n-n_0)/4}\, \text{ for all } p, q\in \mathbb G_n, \text{nearest neighbors}\right\}. \] Thanks to Lemma \ref{lem:N:holder}, a union bound gives for $n\ge n_0$
\begin{align*} P\left(A^c(n,\lambda)\right) &\le C_1 2^{3(n-n_0)} \exp\left(-C_2\frac{\lambda^2 K^2}{\mathscr C_2^2} \varepsilon^2  2^{n_0}2^{(n-n_0)/2} \right)\\
&\le C_1 2^{3(n-n_0)} \exp\left(-C_2 \frac{\lambda^2 K^2}{\mathscr C_2^2\sqrt \alpha} 2^{(n-n_0)/2}\right).
\end{align*}
Let $A(\lambda) = \cap_{n\ge n_0} A(n,\lambda)$. Therefore
\[ P\big(A(\lambda)^c\big) \le \sum_{n\ge n_0} P\big(A(n,\lambda)^c\big) \le \exp\left(-C_3 \frac{\lambda^2 K^2}{\mathscr C_2^2\sqrt\alpha}\right)\]
for some constant $C_3>0$, as long as $\lambda^2/\sqrt{\alpha}\ge \tilde C_3 \mathscr C_2^2$ for some $\tilde C_3$ large enough. Since the left hand side is a probability it is at most $1$. Therefore we can conclude that there exist constants $C_4, C_5>0$ such that for all $\alpha\ge 1$ and $\lambda>0$ we have
\[ P\big(A(\lambda)^c\big)\le C_4 \exp\left(-C_5 \frac{\lambda^2 K^2}{\mathscr C_2^2 \sqrt\alpha} \right).\] 

Now consider a point $(t,x)\in [0,\alpha\varepsilon^4]\times [0,\varepsilon^2]$ which is in a grid $\mathbb G_n$ for some $n\ge n_0$. From arguments similar to page 128 of \cite{spde-utah} we can find a sequence of points $(0,0)= p_0, p_1,\cdots ,p_m =(t,x)$ of points in $\mathbb G_n$ such that each pair $p_{i-1}, p_i$ are nearest neighbors in some grid $\mathbb G_k, \, n_0\le k\le n$, and at most $4$ such pairs are nearest neighbors in any given grid $\mathbb G_k$. Therefore on the event $A(\lambda)$ we have 
\begin{align*}
|\N(t,x)|&\le \sum_{k=1}^m \left \vert \N(p_{j-1})- \N(p_j) \right \vert \\
& \le 4 \sum_{n \ge n_0} \lambda K \varepsilon 2^{-(n-n_0)/4} \le \lambda \varepsilon. 
\end{align*} 
This points $(t, x) \in \mathbb G_n$ are dense in $[0, \alpha \varepsilon^4] \times [0, \varepsilon^2]$, and therefore we have \eqref{eq:N:tail} in the case $\alpha\ge 1$.

 Let us next consider the case $0<\alpha<1$. We divide the interval $[0,\epsilon^2]$ into $\frac{1}{\sqrt{\alpha}}$ intervals each of length $\sqrt{\alpha\epsilon^2}$. A simple union bound and stationarity in $x$ implies that
\begin{align*} P \left(\sup_{\stackrel{0\le t\le \alpha\varepsilon^4}{x\in [0,\varepsilon^2]}} |\N(t,x)|>\lambda \varepsilon\right) &\le \frac{1}{\sqrt \alpha}P \left(\sup_{\stackrel{0\le t\le (\sqrt \alpha\varepsilon^2)^2}{x\in [0,\sqrt\alpha\varepsilon^2]}} |\N(t,x)|>\lambda \varepsilon\right)\\
  & =\frac{1}{\sqrt \alpha}P \left(\sup_{\stackrel{0\le t\le (\sqrt \alpha\varepsilon^2)^2}{x\in [0,\sqrt\alpha\varepsilon^2]}} |\N(t,x)|>\frac{\lambda}{\alpha^{1/4}} (\alpha^{1/4}\varepsilon)\right)\\
& \le \frac{\mf K_1}{\sqrt \alpha} \exp\left(-\mf K_2 \frac{\lambda^2 }{\mathscr C_2^2 \sqrt \alpha} \right).
\end{align*}
This completes the proof the lemma.
\end{proof}

\begin{rem} \label{rem:N:tail:eps}  Suppose the function $\sigma$ in $\mf N$ (see equation \eqref{eq:noise}) satisfies 
\begin{equation*}
|\sigma\big(s, y, \uv(s,y)\big)|\le C_1\varepsilon
\end{equation*}
then the probability in \eqref{eq:N:tail} is bounded above by $\frac{\mf K_1}{1\wedge \sqrt \alpha}\exp\left( - \mf K_2 \frac{\lambda^2}{ C_1^2   {\msr C}^2_2 \varepsilon^2\sqrt\alpha}\right)$ for the same constants $\mf K_1, \mf K_2$ as in \eqref{eq:N:tail}.  This can be proved similarly to the above lemma and will be used later.
\end{rem}

\section{Proof of Proposition \ref{prop:bd}} \label{sec:prop}

\subsection{Proof of Proposition \ref{prop:bd}(a)}
Let $u_1(t,x)$ be the first coordinate of $\uv(t,x)$. Let us define 
\begin{equation*}\label{eq:F:tilde}
\widetilde F_n =\left \{\vert u_1(p_{nj})\vert \le \varepsilon \text{ for all } j \le n_2-2 \right\}.
\end{equation*} 
Since $F_n\subset \widetilde F_n$, it is sufficient to prove
\begin{align} 
 P\Big(\widetilde F_n \Big\vert \bigcap_{k=0}^{n-1} F_k\Big) \le C_4 \exp\left(- \frac{C_5}{(1+\dlip^2)\varepsilon^{2}}\right).
\end{align}

The following lemma on the random variables $N_1(p_{1k})$ (recall that $N_1$ is the first coordinate of $\N$) is crucially used in the proof of the proposition. It shows that the variance of $N_1(p_{1k})$ is of order $\varepsilon^2$, and gives a bound on the decay of correlations between random variables $N_1(p_{1k})$ and $N_1(p_{1k'})$ as $|k-k'|$ increases.  
\begin{lem}\label{lem:N:var}The random variables $N_1(p_{1k})$ are Gaussian with mean $0$. Furthermore there exist constants $C_8, C_9, C_{10}$ depending only on ${\mathscr C}_1, {\mathscr C}_2$ such that 
\begin{align}\label{eq:N:var}
 C_8 \varepsilon^2 \le \textnormal{Var}\big(N_1(p_{1k})\big) &\le C_9 \varepsilon^2\\
\label{eq:N:cov}
\textnormal{Cov}\big(N_1(p_{1k}),\, N_1(p_{1k'})\big) &\le C_{10} t_1 \sup_{0\le t\le 2t_1} \frac{1}{\sqrt{t}} \exp\Big(- \frac{|x_k-x_{k'}|_*^2}{2t}\Big).
\end{align}
Furthermore, if $0<|x_k -x_{k'}|\le \frac12$ and $\theta$ is as in \eqref{chth}, one obtains 
\begin{equation}\label{eq:N:cov:1} \textnormal{Cov}\big(N_1(p_{1k}),\, N_1(p_{1k'})\big)  \le C_{10}\varepsilon^2 \exp\left(-\frac{\theta (k-k')^2}{8}\right).
\end{equation} 
\end{lem}
\begin{proof} It is immediate that the random variables are mean zero Gaussian. As for the covariance
\begin{align}  \label{eq:N:cov:2}
\Cov\big(N_1(p_{1k}),\, N_1(p_{1k'})\big) &\le C{\mathscr C}_2^2 \int_0^{t_1} \int_0^1 G(t_1-s, x_k-y) G(t_1-s, x_{k'}-y) \, dy ds \nonumber\\
&= C{\mathscr C}_2^2 \int_0^{t_1} G(2t_1- 2s, x_k- x_{k'}) \, ds \\
&\le C{\mathscr C}_2^2 t_1 \sup_{0\le t\le 2t_1} G(t, x_k-x_{k'}) \nonumber \\
&\le C{\mathscr C}_2^2 t_1 \sup_{0\le t\le 2t_1} \frac{C}{\sqrt{2\pi t}} \exp\Big(- \frac{|x_k-x_{k'}|_*^2}{2t}\Big). \nonumber
\end{align}
The last inequality follows by the symmetry of $G(t,x)$ in $x$ and by Lemma \ref{lem:G:bd}. The bound for the variance obtained above is $\infty$ which is useless. We instead use the expression after the equality above and Lemma \ref{lem:G:bd} to obtain the upper bound in \eqref{eq:N:var}. The lower bound in \eqref{eq:N:var} can be obtained similarly since the components of $\sigma$ are bounded away from $0$ as well. 

Let us turn our attention back to \eqref{eq:N:cov:2} and consider the situation when $|x_k-x_{k'}| \le \frac 12$. In this case we have $|x_k-x_{k'}|_*= |x_k-x_{k'}|$. Thus 
\begin{align*} \Cov\big(N_1(p_{1k}),\, N_1(p_{1k'})\big) &\le C\cu^2\varepsilon^2 \sup_{0\le t\le 2c_0} \exp\left(- \frac{c_1^2(k-k')^2}{2t}+\frac{\log(1/t)}{2}\right) \\
& =  C\cu^2\varepsilon^2 \sup_{s\ge (2c_0)^{-1}} \exp\left(-\frac12\left[c_1^2(k-k')^2 s -\log s\right]\right). 
\end{align*}
By our choice of $c_1^2=\theta c_0$ with $\theta$ as in \eqref{chth} we see that the maximum is attained at $(2c_0)^{-1}$ for $k\ne k'$. Indeed the expression inside the exponential is decreasing in the interval $[(2c_0)^{-1}, \infty)$. 
Moreover the quantity in brackets inside the exponential is at least $\theta(k-k')^2/4$ as long as $\theta \ge 4 \log [(2c_0)^{-1}]$, which we have assumed in \eqref{chth}. 
This proves \eqref{eq:N:cov:1} and completes the proof of the lemma.
\end{proof}

\begin{proof}[Proof of Proposition \ref{prop:bd}(a) ({\em ${\mathscr D} \equiv 0$ -- The Gaussian case}, i.e. deterministic $\sigma$)] By the Markov property, it is enough to show that $$P(\widetilde F_1) \le C_0 \exp(-C_1\varepsilon^{-2}),$$ for constants $C_0, C_1$ depending only on $\cl,\cu$, if we start with a deterministic initial profile $\uv_0$ satisfying $|\uv_0(x_j)|\le \varepsilon$ for all $j\le n_2-2$. Let $\widetilde H_{-1}=\Omega$ and for $j\ge 0$ define the events 
\[\widetilde H_j =\{ |u_1(p_{1k})|\le \varepsilon \text{ for all } k \le j\}.\]
We will show that  
\begin{equation}\label{eq:g} P(\widetilde H_j \vert \widetilde H_{j-1}) \le \eta \, \text{ for all } 0\le j\le n_2-2,\end{equation}
for some constant $\eta=\eta(\cl,\cu)<1$. Since 
$n_2=[c_1^{-1}\varepsilon^{-2}]$, we have that
\begin{align*}
P(\widetilde F_1) &= \prod_{j=0}^{n_2-2} P(\widetilde H_j \vert \widetilde H_{j-1}) \le \eta^{[c_1^{-1} \varepsilon^{-2}]}
\end{align*}
as required. Let us therefore turn our attention to proving \eqref{eq:g}. Now 
\[ u_1(p_{1k}) = \big[G_{t_1}(\uv_0)(x_k)\big]_1 + N_1(p_{1k}).\]
The term $\big[G_{t_1}(\uv_0)(x_k)\big]_1$ is the first component of the deterministic term in \eqref{eq:mild}, while $N_1(p_{1k})$ are mean zero Gaussian random variables. To obtain \eqref{eq:g} we will show the existence of some $0\le \eta<1$ such that 
\begin{equation}\label{eq:g:2}
P\big(|u_1(p_{1j})|\le \varepsilon \,\vert\, \mathcal{G}_{j-1}\big) \le \eta \text{ for all } 0\le j\le n_2-2,
\end{equation}
where $\mathcal G_{j}$ is the $\sigma$-algebra generated by the random variables $N_1(p_{1k}),\, 0\le k\le j$. The inequality \eqref{eq:g:2} gives a uniform bound on the probability of the event $|u_1(p_{1j}|\le \varepsilon$, given every realization of the random variables $N_1(p_{1k}),\, 0\le k\le j-1$. In particular it gives the same bound on $P(|u_1(p_{1j})|\le \varepsilon \, \vert \,\widetilde H_{j-1})$. 

Therefore let us turn our attention to proving \eqref{eq:g:2}. 

This will be achieved by producing a uniform (in $j$) lower bound of order $C\varepsilon^2$ on the conditional variance of $u_1(p_{1j})$ given $\mathcal G_{j-1}$, where $C$ depends only on $\cl, \cu$. This will imply that the conditional distribution of $u_1(p_{1j})$ is sufficiently {\it spread out} and that the event $|u_1(p_{1j})|>\varepsilon$ has nonvanishing probability. 

By general properties of Gaussian random vectors we can decompose 
\begin{equation}\label{eq:n:decomp} u_1(p_{1j}) = \big[G_{t_1}(\uv_0)(x_{k})\big]_1+X + Y,\end{equation}
where $X$ is the conditional expectation of $ N_1(p_{1j})$ given $\mathcal G_{j-1}$. Furthermore
\begin{equation} \label{eq:X:decomp}
 X= \sum_{k=0}^{j-1} \beta_k^{(j)} N_1(p_{1k})
 \end{equation}
for some coefficients $\beta_l^{(j)}$. The variance of the random variable $Y= N_1(p_{1j})-X$ is precisely the conditional variance of $N_1(p_{1j})$ given $\mathcal G_{j-1}$, and is {\it also} the conditional variance of $u_1(p_{1j})$ given $\mathcal G_{j-1}$.

Let us use the notation $\text{SD}$ to denote the standard deviation of a random variable. By Minkowski's inequality 
\begin{align*}
\text{SD}(X)&\leq \sum_{k=0}^{j-1}|\beta_k^{(j)}|\cdot  \text{SD}\big( N_1(p_{1k})\big).
\end{align*}
We will show in the following lemma that $\sum_{k=0}^{j-1} |\beta_k^{(j)}|$ can be made less than $1/2$ by our choice of $\theta$. In particular for this choice of $\theta$ the standard deviation of $X$ is less than one half the standard deviation of $N_1(p_{1j})$. Therefore, for this choice of $\theta$
\[ \text{SD}\big( N_1(p_{1k})\big) \le \text{SD}(X)+\text{SD}(Y) \le \frac{\text{SD}\big( N_1(p_{1k})\big)}{2} +\text{SD}(Y). \]
From \eqref{eq:N:var} the variance of $N_1(p_{1j})$ is bounded below by $C_8\varepsilon^2$. We have thus shown that the conditional variance of $N_1(p_{1j})$ given $\mathcal G_{j-1}$, which is also the variance of $Y$, is uniformly (in $j$) bounded below by $C_{11}\varepsilon^2$. Recall that the conditional variance of $N_1(p_{1j})$ given $\mathcal G_{j-1}$ is also the conditional variance of $u_1(p_{1j})$ given $\mathcal G_{j-1}$. This implies
\[ P\big(|u_1(p_{1j})|\le \varepsilon \,\vert\, \mathcal{G}_{j-1}\big)\le  \eta\]
for some $\eta<1$ uniformly in $j$. Indeed, for a Gaussian random variable $Z\sim N(\mu, \sigma^2)$ and any $a>0$ the probability $P(|Z|\le a)$ is maximized when $\mu=0$. Therefore
\begin{align*}
P\big(|u_1(p_{1j})|\le \varepsilon \,\vert\, \mathcal{G}_{j-1}\big) &\le P\left(\left|N(0,1)\right|\le \frac{\epsilon}{\sqrt{\text{Var}\left(u_1(p_{1j})\,\vert\, \mathcal{G}_{j-1} \right)}}\right) \\
&\le P\left(|N(0,1)|\le \frac{1}{\sqrt{C_{11}}}\right).
\end{align*}
This completes the proof of the proposition.
\end{proof}

The only ingredient left in the proof of Proposition \ref{prop:bd} in the Gaussian case is the following lemma. 
\begin{lem} Recall the coefficients $\beta_k^{(j)}$ from \eqref{eq:X:decomp}. By choosing $\theta$ as in \eqref{chth}
\[ \sum_{k=0}^{j-1} |\beta_k^{(j)} |\le \frac12\; \text{ for all } 0 \le j \le n_2-2. \]
\end{lem}
\begin{proof}
The random variable $Y$, defined in \eqref{eq:n:decomp}, has the nice property that it is independent of $\mathcal G_{j-1}$, and therefore
\[ \Cov(Y, N_1(p_{1k}))=0 \text{ for } k=0,1\cdots j-1.\] 
It follows that 
\begin{equation} \label{eq:cov:N} \Cov\big(N_1(p_{1j}), N_1(p_{1l})\big)= \sum_{k=0}^{j-1} \beta_k^{(j)} \Cov\big(N_1(p_{1k}), N_1(p_{1l})\big)\, \text{ for all } l=0,1,\cdots, j-1.\end{equation}
Consider now the vector $\boldsymbol \beta=(\beta_0^{(j)}, \beta_1^{(j)},\cdots, \beta_{j-1}^{(j)})^T$ and 
\[ \mf y = \Big(\Cov\big(N_1(p_{1j}), N_1(p_{10})\big), \cdots, \Cov\big(N_1(p_{1j}), N_1(p_{1,j-1})\big)\Big)^T. \]
Let $\mf S= \left(\left(\Cov\big(N_1(p_{1k}), N_1(p_{1l})\big)\right)\right)_{0\le k, l\le j-1}$ be the covariance matrix. The system \eqref{eq:cov:N} can be succinctly written as $\mf y =\mf S\boldsymbol \beta$, whence 
\[ \boldsymbol \beta = \mf S^{-1} \mf y.\] 
Denote by $\|\cdot\|_{1}$ the $\ell_1$ norm on $\mf R^k$, and by $\|\cdot \|_{1,1}$ the matrix norm induced on $j\times j$ matrices by the $\|\cdot\|_{1}$ norm, that is for a matrix $\bf A$
\[ \|{\bf A}\|_{1,1} = \sup_{\bf x\ne \bf 0} \frac{\|{\bf Ax}\|_1}{\|{\bf x}\|_1}.\]
It can be shown that $\|{\bf A}\|_{1,1}= \max_j \sum_{i=1}^n |a_{ij}|$ (see page 259 of \cite{bhim-rao}). Therefore we have 
\begin{equation} \label{eq:beta:bd}
\|\boldsymbol{\beta}\|_{1}\le \|\mf S^{-1}\|_{1,1} \|{\bf y}\|_{1}.
\end{equation} 
Now we write ${\bf S}={\bf D T D }$ where $\bf D$ is a diagonal 
matrix with diagonal entries \\
$\sqrt{\Var\big(N_1(p_{1k})\big)}$, and 
$\bf T$ is the matrix of correlation coefficients 
\[ t_{kl} = \frac{\Cov\big(N_1(p_{1k}), N_1(p_{1l})\big)}{\sqrt{\Var\big(N_1(p_{1k})\big)}\sqrt{\Var\big(N_1(p_{1l})\big)}},\; 0\le k, l \le j-1.\]
Therefore ${\bf S}^{-1}= {\bf D}^{-1}{\bf T}^{-1}{\bf D}^{-1}={\bf D}^{-1}({\bf I}-{\bf A})^{-1} {\bf D}^{-1}$ for some matrix ${\bf A}$. Thanks to \eqref{eq:N:var} and \eqref{eq:N:cov:1} we can make $\|{\bf A}\|_{1,1}<\frac 13$ by choosing $\theta$ as in the second inequality in \eqref{chth}. 
Therefore \[
\|({\bf I}-{\bf A})^{-1}\|_{1,1}\le \frac{1}{1-\|{\bf A}\|_{1,1}} \le \frac32
\]
and so $\|{\bf S}^{-1}\|_{1,1} \le \|{\bf D}^{-1}\|_{1,1}\cdot \|{\bf T}^{-1}\|_{1,1}\cdot \|{\bf D}^{-1}\|_{1,1} \le 2C_8^{-1}\varepsilon^{-2}$, where $C_8$ is as in \eqref{eq:N:var}. Note in obtaining an upper bound on $\|{\bf D}^{-1}\|_{1,1}$, we have crucially used the lower bound in in \eqref{eq:N:var}, which in turn is based on the assumption that the components of $\sigma$ are bounded away from zero..
Substituting the above bounds into \eqref{eq:beta:bd} we obtain
\begin{equation} \label{eq:beta} \|\boldsymbol \beta\|_{1}\le  2C_8^{-1} \varepsilon^{-2} \|{\bf y}\|_1,
 \end{equation}
which can be made less than $1/2$ by choosing $\theta$ as in \eqref{chth}.  \end{proof}

\begin{proof}[Proof of Proposition \ref{prop:bd}(a) ({\em The general case} i.e general $\sigma$)] By the Markov property it is enough to show the bound $P(F_1) \le C_0 \exp\left(-\frac{C_1}{(1+\dlip^2)\varepsilon^{2}}\right)$, for constants $C_0, C_1$ depending only on $\cl, \cu$, starting with a deterministic initial profile $\uv_0$ with $|\uv_0(x_j)|\le \varepsilon$ for all $j\le n_2-2$. 

For a point ${\bf z} \in \R^d$ define 
\begin{equation*}
f_{\varepsilon}\big({\bf z}\big) = \begin{cases} {\bf z}, &  |{\bf z}|\le \varepsilon, \\
\frac{\varepsilon}{|{\bf z}|}{\bf z}, &|{\bf z}|>\varepsilon.\end{cases}
\end{equation*}
In particular $|f_\varepsilon({\bf z})|\le \varepsilon$. We consider the equation 
\begin{equation*}
\partial_t\vv(t,x)=\frac{1}{2}\,\partial_x^2\vv(t,x)+\sigma\left(t, x, f_\varepsilon\left(\vv(t,x)\right)\right)\cdot\dot{\mf W}(t,x)
\end{equation*}
with initial profile $\uv_0$. It is clear that as long as $|\uv(t,x)|\le \varepsilon$ for all $x$, we have $\uv(t,\cdot)=\vv(t,\cdot)$. Therefore it enough to prove the proposition for $\vv$. 

We compare $\vv$ with $\vv_g$ where 
\begin{equation}\label{eq:vg} 
\partial_t\vv_g 
= \frac12\,\partial_x^2\vv_g + \sigma\big(t,x, f_\varepsilon\left(\uv_0(x)\right) \big)\cdot\dot \W (t,x)
\end{equation}
starting with the same initial profile $\uv_0$. We decompose $\vv(t,x)= \vv_g(t,x) + \D(t,x)$, with 
\[ \D(t,x) = \int_0^t \int_0^1 G(t-s, x-y) \left[\sigma\big( s, y, f_\varepsilon(\vv(s,y))\big) - \sigma\big( s, y, f_\varepsilon(\uv_0(y))\big)\right]\cdot \W(dyds).\]
Define
\[ H_j = \{|\vv(p_{1j})|\le \varepsilon \}, \]
so that $F_1=\cap_{j=0}^{n_2-2} H_j$. Define also the events 
\begin{align*}
A_{1,j} &= \{|\vv_g(p_{1j})|\le 2\varepsilon \}, \text{ and }\\
A_{2,j}&=  \{ |\D(p_{1j})|> \varepsilon \}.
\end{align*}
If $|\D(p_{1j})|\le \varepsilon$ and $|\vv_g(p_{1j})|>2\varepsilon$, we must
have $|\vv(p_{1j})|>\varepsilon$. As a consequence,
\[ H_j \subset A_{1,j} \cup A_{2,j}. \]
Therefore
\[ P(F_1) \le P\left( \bigcap_{j=0}^{n_2-2} \left[A_{1,j} \cup A_{2,j} \right]\right). \] 
The intersection can be expanded as the union of various events. One of the terms in the union is
\[  A_{1,0}\cap A_{1,1} \cap A_{1,2}\cap \cdots \cap A_{1, n_2-2}. \]
The remaining events in the union looks like this: 
\[ A_{1,0}\cap A_{1,1} \cap A_{1,2}\cap \cdots \cap A_{1, k-1} \cap A_{2,k} \cap \cdots.\]
That is, it will involve a run of  $\{A_{1,j}: 0 \leq j \leq k-1\}$ and then by $A_{2,k}$  followed by the intersection with other sets. We collect all these events which have the same time $k$ of the first appearance of an $A_{2,j}$. The probability of the union of all these sets is dominated by  $$ P(A_{1,0} \cap A_{1,1}\cap A_{1,2} \cdots \cap A_{1, k} \cap A_{2,k}).$$ Thus we obtain the upper bound 
\begin{align*}
P(F_1) & \le P \left( \bigcap_{j=0}^{n_2-2} A_{1,j}\right) + \sum_{k=0}^{n_2-2} P\left( A_{1,0}\cap A_{1,1} \cap A_{1,2} \cdots \cap A_{1, k-1} \cap A_{2,k}\right) \\
& \le P \left( \bigcap_{j=0}^{n_2-2} A_{1,j}\right) + \sum_{j=0}^{n_2-2} P(A_{2,j}).
\end{align*}
The argument for the Gaussian case shows that 
\begin{equation} \label{eq:a1}P \left( \bigcap_{j=0}^{n_2-2} A_{1,j}\right) = P\Big(\vert \vv_g(p_{nj})| \le 2\varepsilon \text{ for all } j\le n_2-2 \Big) \le C_2 \exp\left(-C_3\varepsilon^{-2}\right),
\end{equation}
for constants $C_2, C_3$ depending only on $\cl, \cu$. By the Lipschitz assumption on $\sigma$ we have 
\[ \big \vert \sigma\big(t,x, f_\varepsilon(\vv(t, x))\big) -\sigma\big(t,x, f_\varepsilon(\uv_0(x))\big) \big \vert \le 2\dlip \varepsilon.\]
An application of \eqref{eq:N:tail} with $\mf N$ replaced by $\D$ (see Remark \ref{rem:N:tail:eps}) gives 
\[ P(A_{2,j}) \le {\mf K}_1\exp\left(-\frac{{\mf K}_2}{4\cu^2\dlip^2\varepsilon^{2}\sqrt c_0}\right).\]
Then by adjusting $\mf K_1,\mf K_2$ and using our bound 
\eqref{eq:est-n2} on $n_2$, we get
\begin{align*} 
\sum_{j=0}^{n_2-2} P(A_{2,j})
&\leq (n_2-2){\mf K}_1\exp\left(-\frac{{\mf K}_2}{4\cu^2\dlip^2\varepsilon^{2}\sqrt c_0}\right) \\
&\leq \frac{1}{c_1\varepsilon^2}
  {\mf K}_1\exp\left(-\frac{{\mf K}_2}{4\cu^2\dlip^2\varepsilon^{2}\sqrt c_0}\right) \\
&\leq  {C}_4\exp\left(-\frac{{C}_5}{8\cu^2(1+\dlip^2)\varepsilon^{2}\sqrt c_0}\right)
\end{align*}
when $\varepsilon$ is small enough, for some constants $C_4, C_5$ dependent on $d, \cu$. Combining the above bound with \eqref{eq:a1} 
completes the proof of the proposition for the general case.  
\end{proof}

\subsection{Proof of Proposition \ref{prop:bd}(b)}
We begin by stating  the Gaussian correlation inequality which we will need.

\begin{lem} \label{lem:gcorr}
For any convex symmetric sets $K,L$ in ${\mf R}^d$ and any centered Gaussian 
measure $\mu$ on ${\mf R}^d$ we have
\begin{equation*}
\mu(K\cap L)\ge \mu(K)\mu(L).
\end{equation*}
\end{lem}
\begin{proof} cf. \cite{roye} and \cite{lata-matl}.
  \end{proof}

By the Markov property of the solution $\uv(t, \cdot)$ to \eqref{eq:she}  (see page 247 in \cite{dapa-zabc}), the behavior of $\uv(t,\cdot)$ in the interval $I_n$ depends only on the profile $\uv(t_n,\cdot)$. Therefore it is enough to prove that there
are constants $C_1,C_2 >0$ such that 
\begin{equation}\label{eq:a_0}
  P(A_0)\ge C_1 \exp (-C_2\varepsilon^{-2}),
\end{equation} for constants $C_1, C_2$ depending only $\cl,\cu$ when $\uv_0$ satisfies $|\uv_0(x)|\le \varepsilon/3$. 

 \begin{proof}[Proof of Proposition \ref{prop:bd}(b) ({\em ${\mathscr D} \equiv 0$ -- The Gaussian case} i.e. deterministic $\sigma$)]    
Consider the event
\be 
\label{eq:b0} B_0= \left\{ |\uv(t_{1}, x)|\le \frac\varepsilon 6 \;\;\forall x, \text{ and } |\uv(t,x)|\le \frac{2\varepsilon}{3} \;\;\forall t\in I_0, \, x\in [0,1] \right\}. \ee
We shall prove below  
a lower bound of the form \eqref{eq:a_0} with the event $A_0$ replaced by $B_0$. Since $B_0\subset A_0$, this implies \eqref{eq:a_0}.

We employ a standard technique in large deviation theory to obtain lower bounds on probabilities of atypical events. That is, we construct a measure $Q$, absolutely continuous with respect to $P$, under which the event $A_0$ is likely. Once we have done this, we next control the Radon Nikodym derivative $\frac {dQ}{dP}$. Let us denote
\[ G_t (\uv_0)(x) = \int_0^1 G(t, x-y) \uv_0(y) dy,\]
the space convolution of $G(t,\cdot)$ and $\uv_0$. Consider the measure $Q_t$ given by
\begin{equation*}
\frac{dQ_t}{dP_t} = \exp \left( Z_{t_1}^{(1)}-\frac12 Z_{t_1}^{(2)}\right)
\end{equation*}
 where 
 \begin{align*}
 Z_{t_1}^{(1)} &= -\int_0^{t_1}\int_0^1 \sigma^{-1}(s, y) \ \frac{G_s(\uv_0) 
 (y)}{t_1} \, \W(dy ds),\\
 Z_{t_1}^{(2)} &= \int_0^{t_1}\int_0^1 \left| \sigma^{-1}(s, y) \ \frac{G_s(\uv_0) (y)}{t_1} \right|^2 \, dy ds.
 \end{align*}
By Lemma \ref{lem:girsanov} 
 \[ \dot {\widetilde{\W}}(s,y) := \dot {\W}(s,y) + \sigma^{-1}(s,y)\frac{G_s
 (\uv_0)(y)}{t_1} \]
 is a white noise under the measure $Q_t$. We now write $\uv(t,x)$ as
\begin{equation}\label{eq:u:Q}
\begin{split}
\int_{0}^{1}&G(t,x-y)\,\mf u_0(y)\,dy  \\
&\;+\int_{0}^{t}\int_{0}^{1}G(t-s,x-y)\,\sigma\left(s, y 
\right)\,\left[{\widetilde \W}(dyds)- \sigma^{-1}(s,y)\,\frac{G_s(\uv_0) (y)}{t_1}\, dsdy\right]  \\
&= \left(1-\frac{t}{t_1}\right) G_t(\uv_0)(x) +\int_0^t \int_0^1G(t-s,x-y)\,\sigma\left(s, y 
\right) \, \widetilde{\W}(dy ds)
\end{split}
\end{equation} 
The first term is equal to $0$ at time $t_1$ and 
\begin{equation}
\label{eq:first:term}
 \left|\left(1-\frac{t}{t_1}\right) G_t\uv_0(x) \right| \le \frac{\varepsilon}{3},\quad x\in [0,1],\;t\le t_1 . 
\end{equation}
 Denote the second term by 
 \[\widetilde \N(t,x)= \int_0^t \int_0^1G(t-s,x-y)\,\sigma\left(s, y 
\right) \, \widetilde{\W}(dy ds).\]
 Since $\dot {\widetilde \W} $ is a white noise under $Q_t$ we can apply \eqref{eq:N:tail} to conclude 
 \begin{equation*} \begin{split}
 { Q_t \left(\sup_{\stackrel{0\le t\le t_1}{x\in [0, c_0\varepsilon^2]}} |\widetilde \N(t,x)|>\frac \varepsilon 6 \right)} &=  { Q_t \left(\sup_{\stackrel{0\le t\le c_0^{-1}(c_0\varepsilon^2)^2}{x\in [0, c_0\varepsilon^2]}} |\widetilde \N(t,x)|>(6\sqrt c_0)^{-1}\cdot( \sqrt c_0\varepsilon )\right)}\\
& \le  { \mf K_1\exp\left( - \frac{\mf K_2}{36 \mathscr C_2^2\sqrt c_0 }\right),}
 \end{split}
\end{equation*}
where we used $\alpha=c_0^{-1}$ and $\lambda= (6\sqrt c_0)^{-1}$ in \eqref{eq:N:tail}.  An application of Lemma \ref{lem:gcorr} gives the following.  
\begin{align*} 
Q_t \left(\sup_{\stackrel{0\le t\le t_1}{x\in [0,1]}} |\widetilde \N(t,x)|\le \frac \varepsilon 6 \right) 
&\ge  { Q_t \left(\sup_{\stackrel{0\le t\le t_1}{x\in [0,c_0\varepsilon^2]}} |\widetilde \N(t,x)|\le \frac \varepsilon 6 \right)^{\frac{1}{c_0\varepsilon^2}} }\\
& \ge { \left[ 1- \mf K_1\exp\left( - \frac{\mf K_2}{36\mathscr C_2^2\sqrt c_0 }\right) \right]^{\frac{1}{c_0\varepsilon^2}}  }
\end{align*}
We observed earlier that the first term in \eqref{eq:u:Q} is bounded by $\varepsilon/3$ and is $0$ at time $t_1$. Therefore 
\[Q_t(B_0) \ge Q_t \left(\sup_{\stackrel{0\le t\le t_1}{x\in [0,1]}} |\widetilde \N(t,x)|\le \frac \varepsilon 6 \right). \]
We finally compare $P_t(A_0)$ with $Q_t(A_0)$. To do this we first observe
\begin{align*}
 E\left(\frac{dQ_t}{dP_t}\right)^2 &= \exp \left(\int_0^{t_1}\int_0^1 \left|
 \sigma^{-1}(s, y) \ \frac{G_s(\uv_0)(y)}{t_1} \right|^2 \, dy ds\right) \\
 &\le {\exp \Big(\frac{1}{c_0\mathscr C_1^2\varepsilon^{2}}\Big) } 
 \end{align*}
The inequality is a consequence of the bound $\max_{|\mf x|=1} | \sigma^{-1} \mf x|^2 \le \lambda_1^{-2}\le \mathscr C^2$, where $\lambda_1$ is the smallest eigenvalue of $\sigma$. By the Cauchy-Schwarz inequality
\begin{align*}
Q_t(B_0) &\le \sqrt{E\left(\frac{dQ_t}{dP_t}\right)^2 } \cdot \sqrt{P_t(B_0)},
\end{align*} 
and from this it follows
\begin{equation} 
\label{eq:b_0} 
P(B_0) \ge {\exp \Big(-\frac{1}{c_0\mathscr C_1^2\varepsilon^{2}}\Big) \exp\left(\frac{2}{c_0\varepsilon^2} \log \left[ 1- \mf K_1\exp\left( - \frac{\mf K_2}{36\mathscr C_2^2\sqrt c_0 }\right)\right]\right)},
\end{equation}
where $\mf K_1, \, \mf K_2$ are the constants appearing in \eqref{eq:N:tail}. With our choice of $c_0$ in \eqref{chc0}, this completes the proof of the proposition in the case of deterministic $\sigma$.
 \end{proof}

 For the Proof of Proposition \ref{prop:bd}(b) with general $\sigma$ we compare $\uv$ with $\uv_g$, where
\begin{equation}\label{eq:ug} \partial_t \uv_g = \frac12 \partial_x^2 \uv_g + \sigma\big(t,x, \uv_0(x) \big)\dot \W (t,x),\quad t\in [0,c_0\varepsilon^4],\; x\in [0,1],\end{equation}
with the same initial profile $\uv_0$. Recall that we are assuming $|\uv_0(x)|\le \varepsilon/3$ for all $x$. We write
\begin{equation} \label{eq:u:ug} \uv(t,x) =\uv_g(t,x) + \D(t,x),\end{equation}
where 
\[\D(t,x) = \int_0^t \int_0^1 G(t-s, x-y) \left[\sigma\big(s, y, \uv(s,y)\big) - \sigma\big(s, y, \uv_0(y)\big)\right] \W(dyds).\]
Let us define
 \be\label{eq:b0:t} \widetilde B_0= \left\{ |\uv_g(t_{1}, x)|\le \frac\varepsilon 6 \;\;\forall x, \text{ and } |\uv_g(t,x)|\le \frac{2\varepsilon}{3} \;\;\forall t\in I_0, \, x\in [0,1] \right\}. \ee
Since the solution of \eqref{eq:ug} is Gaussian, we can apply \eqref{eq:b_0} applied to $\uv_g$ to obtain 
\begin{equation} \label{eq:tilde:b0} P(\widetilde B_0) \ge {\exp \Big(-\frac{1}{c_0\mathscr C_1^2\varepsilon^{2}}\Big) \exp\left(\frac{2}{c_0\varepsilon^2} \log \left[ 1- \mf K_1\exp\left( - \frac{\mf K_2}{36\mathscr C_2^2\sqrt c_0 }\right)\right]\right)},\end{equation}
{ where $\mf K_1$ and $\mf K_2$ are the constants appearing in \eqref{eq:N:tail}. }

\begin{proof}[Proof of Proposition \ref{prop:bd}(b)(general $\sigma$)]
We will again prove \eqref{eq:a_0}. As discussed in the Gaussian case, we will use the Markov property of the solution $\uv(t, \cdot)$ to \eqref{eq:she} and therefore it is enough to prove that there
are constants $C_1,C_2 >0$ such that 
\begin{equation}\label{eq:a_0again}
  P(A_0)\ge C_1 \exp (-C_2\varepsilon^{-2}),
\end{equation} for constants $C_1, C_2$ depending only $\cl,\cu$ when $\uv_0$ satisfies $|\uv_0(x)|\le \varepsilon/3$. 
  Define 
\[ \tau = \inf\{t: |\uv(t,x) -\uv_0(x)|>2\varepsilon,\, \text{ for some } x\in [0,1]\}.\]
Since $\mid \uv_0(x) \mid \leq \frac{\epsilon}{3}$, we must have $\tau>t_1$ on the event $A_0$. Moreover on the event $\tau>t_1$ we have $\D(t,x) = \widetilde \D(t,x)$ for $t\le t_1$, where 
\[ \widetilde \D(t,x) = {\int_0^t \int_0^1 G(t-s, x-y) \left[\sigma\big( s, y, \uv(s\wedge \tau,y)\big) - \sigma\big(s, y, \uv_0(y)\big)\right] \W(dyds).}\]
Therefore, thanks to the decomposition \eqref{eq:u:ug}, we have
\begin{align*}
P(A_0) &\ge P\left( \widetilde B_0 \cap \left\{ \sup_{\stackrel{0\le t\le t_1}{x\in[0,1]}} \big|\D(t,x)\big|\le \frac\varepsilon 6\right\}\right) \\
&= P\left( \widetilde B_0 \cap \left\{ \sup_{\stackrel{0\le t\le t_1}{x\in[0,1]}} \big|\widetilde\D(t,x)\big|\le \frac\varepsilon 6\right\}\right)\\
&\ge P(\widetilde B_0) - P \left(\sup_{\stackrel{0\le t\le t_1}{x\in[0,1]}} \big|\widetilde \D(t,x)\big|> \frac\varepsilon 6\right). 
\end{align*}
Let us explain the equality above. On the event $\{\tau>t_1\}$ we have $\sup_{x\le 1,t\le t_1} |\D(t,x)|=\sup_{x\le 1,t\le t_1} |\widetilde\D(t,x)|$, whereas on the event $\widetilde B_0\cap \{\tau \le t_1\}$ we have 
\[ \sup_{t\le t_1,\, x\le 1} |\widetilde \D(t,x)| \ge \sup_x |\widetilde \D(\tau,x)|=\sup_x|\D(\tau,x)| \ge \varepsilon,\]
a consequence of $|\uv_g(\tau,x)|\le2\varepsilon/3$ and $|\uv(\tau,x)-\uv_0(x)|>2\varepsilon$ for some $x$ (recall that we are assuming that $|\uv_0(x)|\le \varepsilon/3$ for all $x$).  
Now a union bound gives 
\begin{equation}\label{eq:Dt}
\begin{split}
P\left(\sup_{\stackrel{0\le t\le t_1}{x\in[0,1]}} \big|\widetilde \D(t,x)\big|> \frac\varepsilon 6 \right) 
& \le \frac{1}{\sqrt c_0\varepsilon^2} P\left(\sup_{\stackrel{0\le t\le t_1}{x\in[0,\sqrt c_0\varepsilon^2]}} \big|\widetilde \D(t,x)\big|> \frac\varepsilon 6 \right) 
\\
&{ =\frac{1}{\sqrt c_0\varepsilon^2}  P\left(\sup_{\stackrel{0\le t\le t_1}{x\in[0,\sqrt c_0\varepsilon^2]}} \big|\widetilde \D(t,x)\big|> \frac{1}{6c_0^{1/4}}\cdot c_0^{1/4} \varepsilon\right)  }\\
&\le {\frac{\mf K_1}{\sqrt c_0\varepsilon^2} \exp\left( - \frac{\mf K_2}{144\mathscr D^2  {\msr C}^2_2 \varepsilon^2 \sqrt c_0}\right),}
\end{split}
\end{equation}
where we applied Remark \ref{rem:N:tail:eps} to $\tilde \D$ instead of $\N$. Thus when $\varepsilon$ is small enough we have 
\begin{equation}
\label{eq:D-tilde-estimate}
P\left(\sup_{\stackrel{0\le t\le t_1}{x\in[0,1]}} \big|\widetilde \D(t,x)\big|> \frac\varepsilon 6 \right)  \le {\mf K_1 \exp\left( - \frac{\mf K_2}{288\mathscr D^2\cu^2 \varepsilon^2 \sqrt c_0}\right)}
\end{equation}
Using \eqref{eq:tilde:b0} and \eqref{eq:D-tilde-estimate}  we have 
\begin{equation} \label{eq:a0}
\begin{split}
P(A_0) & \ge{ \exp \Big(-\frac{1}{c_0\mathscr C_1^2\varepsilon^{2}}\Big) \exp\left(\frac{2}{c_0\varepsilon^2} \log \left[ 1- \mf K_1\exp\left( - \frac{\mf K_2}{36\mathscr C_2^2\sqrt c_0 }\right)\right]\right) }  \\
&  \hspace{3cm} {- \mf K_1\exp\left( - \frac{\mf K_2}{288\mathscr D^2  {\msr C}^2_2 \varepsilon^2\sqrt c_0 }\right).}
\end{split}
\end{equation}
Consequently there is a ${\mathscr D}_0(\cl, \cu)$  such that if ${\mathscr D} < {\mathscr D}_0$ then there are constants $C_1, C_2 >0$ depending only on $\cl, \cu$ such that
\begin{equation} \label{eq:fa0}
P(A_0) \ge { C_1\exp \Big(-\frac{C_2}{\varepsilon^{2}}\Big)}.
\end{equation}
This completes the proof of Proposition \ref{prop:bd}(b).
\end{proof}

We can now give the proof of Theorem \ref{th:small-ball}
\subsection{Proof of Theorem \ref{th:small-ball}:} Recall that we are working with $J=1$ and $\g\equiv \mf 0$. For the upper bounds in \eqref{corder} and \eqref{acorder} we consider the set 
\[ F= \bigcap_{n=0}^{n_1-2} F_n.\]
Clearly we have 
\[ \left\{\big|\uv(t,x)\big|\le \varepsilon \text{ for } t\in [0,T],\, x\in [0,1]\right\}\; \subset \;F.\]
From Proposition \ref{prop:bd}(a) we obtain
\[ P(F) = \prod_{n=1}^{n_1-2}P\Big(F_n \Big\vert \bigcap_{k=0}^{n-1} F_k\Big) \le \left[C_4 \exp\left(- \frac{C_5}{(1+\dlip^2)\varepsilon^{2}}\right)\right]^{\frac{T}{c_0\varepsilon^4}},\]
which proves the upper bounds in Theorem \ref{th:small-ball}.

For the lower bound in \eqref{corder} consider the set
\[A:= \bigcap_{n=-1}^{n_1-1} A_n. \]
Clearly we have 
\[ A\; \subset\; \left\{ \big|\uv (t,x)\big|\le \varepsilon,\mbox{ for } 0\le t\le T, \, x\in [0,1]\right \}.\]
From Proposition \ref{prop:bd}(b) we obtain for $\dlip <\dlip_0$ 
\[ P(A) = \prod_{n=-0}^{n_1-1}P\Big(A_n \Big\vert \bigcap_{k=-1}^{n-1} A_k\Big) \ge \left[C_6 \exp \left(- \frac{C_7}{\varepsilon^{2}}\right)\right]^{\frac{T}{c_0\varepsilon^4}},  \]
which proves the lower bound in \eqref{corder}.

To prove the lower bound in \eqref{acorder}, we now take $t_1=\varepsilon^{4+2\delta}$ (note that this is smaller than the earlier value of $t_1=c_0\varepsilon^4$, at least for small $\varepsilon$), and generally $t_n=n\varepsilon^{4+2\delta}$. We consider the events $A_n$ (see \eqref{eq:An}), $B_0$ (see \eqref{eq:b0}) and $\widetilde B_0$ (see \eqref{eq:b0:t}), now {\it with this new value of $t_n$}. We have as before 
\[ P(A_0) \ge P(\widetilde B_0) - P \left(\sup_{\stackrel{0\le t\le t_1}{x\in[0,1]}} \big|\widetilde \D(t,x)\big|> \frac\varepsilon 6\right).\]
Since $t_1<c_0\varepsilon^{4+\delta}$ for small $\varepsilon$, we have the
same lower bound \eqref{eq:tilde:b0} for $P(\widetilde B_0)$:
\be \label{eq:b0t:2} P(\widetilde B_0) \ge \exp \Big(-\frac{1}{c_0\mathscr C_1^2\varepsilon^{2}}\Big) \exp\left(\frac{2}{c_0\varepsilon^2} \log \left[ 1- \mf K_1\exp\left( - \frac{\mf K_2}{36\mathscr C_2^2\sqrt c_0 }\right)\right]\right),\ee
because $\tilde{B}_0$ is  a larger event than the {\it earlier} defined event
with $c_0\varepsilon^4$. By a similar argument as \eqref{eq:Dt} we obtain
\[P\left(\sup_{\stackrel{0\le t\le t_1}{x\in[0,1]}} \big|\widetilde \D(t,x)\big|> \frac\varepsilon 6 \right) 
\le \frac{\mf K_1}{\varepsilon^{2+\frac \delta 2}} \exp\left( - \frac{\mf K_2}{144\mathscr D^2  {\msr C}^2_2 \varepsilon^{2+\frac\delta 2} \sqrt c_0}\right).\]
This is much smaller than the lower bound for $P(\widetilde B_0)$, for small $\varepsilon$. Therefore $P(A_0) \ge C_1\exp(-C_2\varepsilon^{-2})$ for constants $C_1,C_2$ depending on $\cl, \cu$ only, when $\varepsilon$ is small enough. Now with $n_1=T\varepsilon^{-4-\delta}$ we obtain 
\[ P(A) = \prod_{n=0}^{n_1-1}P\Big(A_n \Big\vert \bigcap_{k=-1}^{n-1}
A_k\Big) \ge \left[C_6 \exp \left(- \frac{C_7}{\varepsilon^{2}}\right)\right]^{\frac{T}{\varepsilon^{4+\delta}}},  \]
which gives the lower bound in \eqref{acorder}. \qed

\newcommand{\etalchar}[1]{$^{#1}$}
\providecommand{\bysame}{\leavevmode\hbox to3em{\hrulefill}\thinspace}
\providecommand{\MR}{\relax\ifhmode\unskip\space\fi MR }
\providecommand{\MRhref}[2]{%
  \href{http://www.ams.org/mathscinet-getitem?mr=#1}{#2}
}
\providecommand{\href}[2]{#2}


\begin{thebibliography}{DKM{\etalchar{+}}09}

\bibitem[All98]{allo}
H.~Allouba, \emph{Different types of {SPDE}s in the eyes of {G}irsanov's
  theorem}, Stochastic Anal. Appl. \textbf{16} (1998), no.~5, 787--810.
  \MR{1643116}
  
  \bibitem[BMSS95]{bms95}
  Vlad Bally, Annie Millet, and Marta Sanz-Sol\'e, \emph{Approximation and support theorem in H\"older norm for parabolic stochastic partial differential equations}, Ann. Probab. \textbf{23} (1995), no.~1, 178--222. \MR{1330767}

\bibitem[Bas88]{B88}
Richard~F. Bass, \emph{Probability estimates for multiparameter {B}rownian
  processes}, Ann. Probab. \textbf{16} (1988), no.~1, 251--264. \MR{920269}

\bibitem[Bas95]{B95}
\bysame, \emph{Probabilistic techniques in analysis}, xii+392. \MR{1329542}

\bibitem[DKM{\etalchar{+}}09]{spde-utah}
Robert Dalang, Davar Khoshnevisan, Carl Mueller, David Nualart, and Yimin Xiao,
  \emph{A minicourse on stochastic partial differential equations}, Lecture
  Notes in Mathematics, vol. 1962, Springer-Verlag, Berlin, 2009, Held at the
  University of Utah, Salt Lake City, UT, May 8--19, 2006, Edited by
  Khoshnevisan and Firas Rassoul-Agha. \MR{1500166 (2009k:60009)}

\bibitem[DPZ14]{dapa-zabc}
Giuseppe Da~Prato and Jerzy Zabczyk, \emph{Stochastic equations in infinite
  dimensions}, second ed., Encyclopedia of Mathematics and its Applications,
  vol. 152, Cambridge University Press, Cambridge, 2014. \MR{3236753}

\bibitem[DV75]{DV75}
M.~D. Donsker and S.~R.~S. Varadhan, \emph{Asymptotic evaluation of certain
  {M}arkov process expectations for large time. {I}. {II}}, Comm. Pure Appl.
  Math. \textbf{28} (1975), 1--47; ibid. 28 (1975), 279--301. \MR{386024}

\bibitem[KL93]{KL93}
James Kuelbs and Wenbo~V. Li, \emph{Metric entropy and the small ball problem
  for {G}aussian measures}, J. Funct. Anal. \textbf{116} (1993), no.~1,
  133--157. \MR{1237989}

\bibitem[LaM17]{lata-matl}
Rafa\l Lata\l~a and Dariusz Matlak, \emph{Royen's proof of the {G}aussian
  correlation inequality}, Geometric aspects of functional analysis, Lecture
  Notes in Math., vol. 2169, Springer, Cham, 2017, pp.~265--275. \MR{3645127}

\bibitem[Lot17]{L17}
S.~V. Lototsky, \emph{Small ball probabilities for the infinite-dimensional
  {O}rnstein-{U}hlenbeck process in {S}obolev spaces}, Stoch. Partial Differ.
  Equ. Anal. Comput. \textbf{5} (2017), no.~2, 192--219. \MR{3640070}

\bibitem[LS01]{LS01}
W.~V. Li and Q.-M. Shao, \emph{Gaussian processes: inequalities, small ball
  probabilities and applications}, Stochastic processes: theory and methods,
  Handbook of Statist., vol.~19, North-Holland, Amsterdam, 2001, pp.~533--597.
  \MR{1861734}

\bibitem[Mar04]{Mar04}
A.~Martin, \emph{Small ball asymptotics for the stochastic wave equation}, J.
  Theoret. Probab. \textbf{17} (2004), no.~3, 693--703. \MR{2091556}

\bibitem[RB00]{bhim-rao}
A.~Ramachandra Rao and P.~Bhimasankaram, \emph{Linear algebra}, second ed.,
  Texts and Readings in Mathematics, vol.~19, Hindustan Book Agency, New Delhi,
  2000. \MR{1781860}

\bibitem[Roy14]{roye}
Thomas Royen, \emph{A simple proof of the {G}aussian correlation conjecture
  extended to some multivariate gamma distributions}, Far East J. Theor. Stat.
  \textbf{48} (2014), no.~2, 139--145. \MR{3289621}

\bibitem[RP04]{rp04}
Raina~S. Robeva and Loren~D. Pitt, \emph{On the equality of sharp and germ
  {$\sigma$}-fields for {G}aussian processes and fields}, Pliska Stud. Math.
  Bulgar. \textbf{16} (2004), 183--205. \MR{2070315}

\bibitem[Tal94]{T94}
Michel Talagrand, \emph{The small ball problem for the {B}rownian sheet}, Ann.
  Probab. \textbf{22} (1994), no.~3, 1331--1354. \MR{1303647}

\bibitem[Tal95]{T95}
\bysame, \emph{Hausdorff measure of trajectories of multiparameter fractional
  {B}rownian motion}, Ann. Probab. \textbf{23} (1995), no.~2, 767--775.
  \MR{1334170}

\bibitem[Wal86]{wals}
John~B. Walsh, \emph{An introduction to stochastic partial differential
  equations}, \'{E}cole d'\'et\'e de probabilit\'es de {S}aint-{F}lour,
  {XIV}---1984, Lecture Notes in Math., vol. 1180, Springer, Berlin, 1986,
  pp.~265--439. \MR{876085 (88a:60114)}

\end{thebibliography}
\end{document}